\documentclass{amsart}
\usepackage{amssymb,amsfonts, color,epsf}
\usepackage [cmtip,arrow]{xy}
\xyoption{all}
\usepackage {pb-diagram,pb-xy}
\usepackage{enumerate}
\usepackage{accents}
\usepackage[noautoscale]{youngtab}
\usepackage{subfigure}

\ifx\pdfpageheight\undefined
\usepackage[dvips,colorlinks=true,linkcolor=blue,citecolor=red,%
urlcolor=green]{hyperref}
\usepackage[dvips]{graphicx}
\makeatletter
\edef\Gin@extensions{\Gin@extensions,.mps}
\makeatother
\else
\usepackage[pdftex]{graphicx}
\usepackage[bookmarksopen=false,pdftex=true,breaklinks=true,%
backref=page,pagebackref=true,plainpages=false,%
hyperindex=true,pdfstartview=FitH,colorlinks=true,%
pdfpagelabels=true,colorlinks=true,linkcolor=blue,%
citecolor=red,urlcolor=green,hypertexnames=false%
]%
{hyperref}
\fi
\usepackage{bm}

\usepackage{tikz-cd}
\makeatletter
\tikzset{
	column sep/.code=\def\pgfmatrixcolumnsep{\pgf@matrix@xscale*(#1)},
	row sep/.code   =\def\pgfmatrixrowsep{\pgf@matrix@yscale*(#1)},
	matrix xscale/.code=%
	\pgfmathsetmacro\pgf@matrix@xscale{\pgf@matrix@xscale*(#1)},
	matrix yscale/.code=%
	\pgfmathsetmacro\pgf@matrix@yscale{\pgf@matrix@yscale*(#1)},
	matrix scale/.style={/tikz/matrix xscale={#1},/tikz/matrix yscale={#1}}}
\def\pgf@matrix@xscale{1}
\def\pgf@matrix@yscale{1}
\makeatother

\usepackage{amsmath,amssymb,amsfonts}
\newtheorem{theorem}{Theorem}
\newtheorem{lemma}{Lemma}[section]

\newtheorem{proposition}{Proposition}[section]
\newtheorem{claim}{Claim}[section]
\newtheorem*{claim*}{Claim}

\newtheorem*{theorem*}{Theorem}
\newtheorem*{corollary*}{Corollary}
\theoremstyle{definition}
\newtheorem{definition}{Definition}[section]
\newtheorem{example}{Example}

\newtheorem{notation}{Notation}

\usepackage{algorithm}
\usepackage{algpseudocode}

\usepackage{tikz}

\algnewcommand\algorithmicinput{\textbf{Input:}}
\algnewcommand\INPUT{\item[\algorithmicinput]}
\algnewcommand\algorithmicoutput{\textbf{Output:}}
\algnewcommand\OUTPUT{\item[\algorithmicoutput]}

\algnewcommand\algorithmicproc{\textbf{Procedure:}}
\algnewcommand\PROCEDURE{\item[\algorithmicproc]}
\algnewcommand\algorithmiccomplexity{\textbf{Complexity:}}
\algnewcommand\COMPLEXITY{\item[\algorithmiccomplexity]}

\newlength{\continueindent}
\usepackage{etoolbox}
\makeatletter
\newcommand*{\ALG@customparshape}{\parshape 2 \leftmargin \linewidth \dimexpr\ALG@tlm+\continueindent\relax \dimexpr\linewidth+\leftmargin-\ALG@tlm-\continueindent\relax}
\apptocmd{\ALG@beginblock}{\ALG@customparshape}{}{\errmessage{failed to patch}}
\makeatother

\theoremstyle{remark}
\newtheorem{remark}{Remark}

\theoremstyle{observation}

\definecolor{DarkBlue}{rgb}{0,0.1,0.55}

\numberwithin{equation}{section}

\newcommand {\hide}[1]{}
\newcommand {\sign} {\mbox{\bf sign}}

\newcommand {\junk}[1]{}

\newcommand {\R} {\mathrm{R}}

\newcommand {\D}     {\mbox{\rm D}}

\newcommand {\Sphere}{\mbox{${\bf S}$}}     

\newcommand {\Z}  {\mathbb{Z}}

\newcommand {\Q}         {\mathbb{Q}}

\newcommand {\ZZ} {{\rm Z}}
\newcommand {\RR} {{\mathcal R}}

\newcommand {\Der} {{\rm Der}}

\newcommand {\la}   {{\langle}}
\newcommand {\ra}   {{\rangle}}
\newcommand {\eps} {{\varepsilon}}
\newcommand {\E} {{\rm ext}}

\newcommand {\Sign}      {\mbox{\rm Sign}}

\newcommand{\card}{\mathrm{card}}
\newcommand{\rank}{\mathrm{rank}}

\def\addots{\mathinner{\mkern1mu
		\raise1pt\vbox{\kern7pt\hbox{.}}
		\mkern2mu\raise4pt\hbox{.}\mkern2mu
		\raise7pt\hbox{.}\mkern1mu}}

\newcommand{\HH}  {\mbox{\rm H}}

\newcommand{\x}{\mathbf{x}}
\newcommand{\X}{\mathbf{X}}

\newcommand{\Top}{\mathbf{Top}}

\newcommand{\Simp}{\mathbf{Simp}}

\newcommand{\crit}{\mathrm{crit}}
\newcommand{\Crit}{\mathrm{Crit}}

\begin{document}
	\title[Persistent homology of semi-algebraic sets]
	{
		Persistent homology of semi-algebraic sets
	}
	\author{Saugata Basu}
	\address{Department of Mathematics,
		Purdue University, West Lafayette, IN 47906, U.S.A.}
	\email{sbasu@math.purdue.edu}
	
	\author{Negin Karisani}
	\address{Department of Computer Science, 
		Purdue University, West Lafayette, IN 47906, U.S.A.}
	\email{nkarisan@cs.purdue.edu}

	\date{\textbf{\today}}
	\thanks{
		Basu was  partially supported by NSF grants
		CCF-1618918, DMS-1620271 and CCF-1910441.}
	\begin{abstract}
		We give an  algorithm with singly
		exponential complexity for computing 
		the \emph{barcodes} up to dimension $\ell$ (for any fixed $\ell \geq 0$)
		of the filtration of a given semi-algebraic set by the sub-level sets of a given polynomial.
		Our algorithm is  the first algorithm for this problem with singly exponential complexity, and generalizes the corresponding results
		for computing the Betti numbers up to dimension $\ell$ 
		of semi-algebraic sets with no filtration present.  
	\end{abstract}
	\maketitle
	
	\tableofcontents
	
	\section{Introduction}
	\subsection{Background}
	Let $\R$ be a real closed field and $\D$ an ordered domain contained in $\R$ which we fix for the rest of the paper.
	The algorithmic problem of computing the ranks of the homology groups of a given semi-algebraic set described by a quantifier-free formula, whose atoms are of the form $P > 0, P \in \D[X_1,\ldots,X_k]$, as input has attracted a lot of attention over the years. 
	\footnote{Here and everywhere else in the paper all homology groups considered are with rational coefficients.}
	Closed and bounded semi-algebraic subsets $S \subset \R^k$ are semi-algebraically triangulable -- and moreover given a description of the semi-algebraic set by a quantifier-free formula, such a triangulation can be effectively computed with complexity 
	(measured in terms of the number of polynomials appearing in the description and their degrees) which is doubly exponential in $k$. 
	Together with standard algorithms of linear algebra, this gives an algorithm for computing all the Betti numbers of a given semi-algebraic set with doubly exponential complexity. 
	
	Classical bounds coming from Morse theory \cite{OP,T, Milnor2, B99,GV07} gives singly exponential bounds on the Betti numbers of semi-algebraic sets. More precisely, if $S \subset \R^k$ is a semi-algebraic set defined by a quantifier-free formula involving 
	$s$ polynomials of degree at most $d$, then the sum of the 
	Betti numbers of $S$ is bounded by $(O(sd))^k$. 
	Additionally, it is known that the problem of computing the zero-th Betti number of $S$ (i.e. the number of semi-algebraically components of $S$)  (using ``roadmap'' algorithms \cite{Canny93a,GV92,BPR99}), as well as the 
	problem of computing the Euler-Poincar\'e characteristic of $S$
	(using Morse theory \cite{B99,BPR-euler-poincare}),
	both admit single exponential complexity algorithms. This led to a search for singly exponential complexity algorithm for computing the higher Betti numbers as well. The current state of the art is that
	there exists singly exponential algorithm for computing the first 
	$\ell$ Betti numbers of semi-algebraic sets for each fixed $\ell \geq 0$ \cite{Bas05-first,BPRbettione}.
	
	In this paper we study the algorithmic complexity of computing a finer 
	topological invariant of a given semi-algebraic set $S$ than its Betti numbers -- namely, \emph{the barcode of a filtration} of the  given semi-algebraic set.  
	This finer invariant (unlike the Betti numbers) has both a discrete as well as a continuous part and is attached to a  filtration of a semi-algebraic set by the sub-level sets of a semi-algebraic function. The problem reduces to the problem of 
	computing the Betti numbers in the case when the given filtration is trivial.

\subsection{Persistent Homology}
	One of the recent developments in the area of applied topology is the introduction of the notion
	of \emph{persistent homology of filtrations}.
	We initiate in this paper the study of the algorithmic problem of computing the \emph{persistent homology groups} (cf. Definition~\ref{def:persistent}) 
	of filtrations of semi-algebraic sets by polynomial functions.
	Persistent homology is a central object in the emerging field of topological data analysis
	\cite{dey2022computational, Weinberger_survey, Ghrist}, but has also found applications
	in diverse areas of mathematics and computations as well (see for example \cite{Ellis-King,Manin-Marcolli}). 
	
	One can associate persistent homology
	groups to any filtration of topological spaces, and they generalize ordinary homology groups of a space $X$ -- which
	corresponds to the trivial (i.e. constant) filtration on $X$.
	To the best of our knowledge the algorithmic problem of computing persistent homology groups of semi-algebraic sets
	equipped with a filtration by the sub-level sets of a polynomial (or more generally continuous semi-algebraic functions) have not been considered from an algorithmic viewpoint. 
	The output of a persistent homology computation is usually expressed in the form of ``barcodes'' \cite{Ghrist}.	

	We will define barcodes for semi-algebraic filtrations precisely later (see Definition~\ref{def:barcode_multiplicity}). A basic example which is a starting point of persistent homology theory in the field of topological data analysis is the following one.

\subsubsection{\v{C}ech complex of a finite set of points}	
	 Let $X$ be a (finite) subset of $\R^k$ (with its Euclidean metric).  In practice, $X$ may consist of a finite set of points  (often  called ``point-cloud data'')
which approximates some subspace or sub-manifold $M$ of $\R^k$. The topology (in particular, the homology groups) of the manifold $M$ is not reflected in the set
of points $X$ (which is a discrete topological space under the subspace topology induced from that of $\R^k$). 
Now for $r \geq 0$, let $X_r$ denote the union of closed Euclidean balls,
$\overline{B_k(x,r)}$, of radius $r$ centered at the points $x \in X$. 
Notice that each $X_r$ is a semi-algebraic set indexed by $r$.
In particular, $X_0 = X$.
Also, for $0 \leq r \leq r'$, we have that $X_r \subset X_{r'}$. Thus,
$(X_r)_{r \geq 0}$ is an increasing family of semi-algebraic sets
indexed by $r \geq 0$. Thus, this is an example of a semi-algebraic filtration (see Remark~\ref{cech-example}). The main rationale for considering this filtration is that
nerve complex of the family of convex sets $\overline{B_k(x,r)}, x \in X$
approximates homotopically the underlying manifold $M$, and
each homology class of $M$ would show up in the homology of $X_r$ for 
some values of $r$. The barcode of the filtration $(X_r)_{r \geq 0}$ is a tool for filtering out spurious homology (noise) from
that which genuinely reflects the topology of $M$ (see \cite{Ghrist,Edelsbrunner_survey}).
The barcode of the above filtration thus plays 
an important role in
topological data analysis. In particular they capture information about the homology of the underlying manifold $M$.
It also serves as a ``signature'' for topological data
	(such as point cloud data). In the semi-algebraic world they play a similar role -- for example, as a measure of topological similarity
	of two given semi-algebraic sets which is much finer (because of the 
	presence of the continuous parameters) to just the sequence of Betti numbers.
    
    As stated earlier, the main goal in this paper is to design an efficient (singly exponential complexity algorithm) that takes as input a quantifier-free formula
	describing a closed semi-algebraic set $S \subset \R^k$ 
	as well as a polynomial $P \in \R[X_1,\ldots,X_k]$, and
	outputs the  barcodes 
	up to dimension $\ell$ for some fixed $\ell \geq 0$ 
	of the filtration of $S$ by the sub-level sets of the function $P$ on $S$, thereby generalizing the algorithm in \cite{Bas05-first} for computing the first $\ell$ Betti numbers with a similar complexity.
	There are several intermediate steps needed to achieve this goal.
	These intermediate steps have been used recently in other applications (that we mention in Section~\ref{subsec:contributions} below)
	and hence could be of independent interest. We outline them below.
	
\subsection{Summary of the main contributions}
\label{subsec:contributions}
	We summarize the  main contributions of the paper as follows.
	\begin{enumerate}[1.]
	    \item We reformulate the definition of barcodes 
	     in order to treat continuous as well as finite filtrations in a uniform manner. This is important in the current application since we consider filtrations of semi-algebraic sets by polynomial functions which are by nature examples of continuous filtrations (since they are indexed by $\R$). 
	    However, we show that the barcode of this continuous filtration is equal  to another finite one (see Propositions~\ref{prop:finite} and \ref{prop:G2}). In order for such an equality to make sense it is important that persistent homology of a filtration should be defined in a uniform way for arbitrary ordered index set. It is possible to have a completely categorical description of persistent homology which applies to very general filtration \cite{Bubenik-Scott}. We avoid categorical language and give an \emph{elementary definition of barcodes}  directly in terms of sub-quotients of homology groups 
	    (see Definition~\ref{def:barcode_multiplicity}). 
	    We remark here that the
basic theory of persistent homology with real parameters
in the multi-persistence setting was developed independently (in different ways) in \cite{Kashiwara2018}
and in \cite{miller2020essential,miller2020homological}. However,
we prefer to give a self-contained description which applies directly to the
the one-dimensional semi-algebraic setting and is suitable from our algorithmic view-point.

\hide{	    
	    which we believe could be useful in other applications as well.  
	    For example, the definition given in our paper is used in a crucial way in \cite{Basu-Cox} to define \emph{harmonic barcodes}  (which carry more information than the classical barcodes). In this application it is very important to identify the persistent homology spaces with certain subspaces
	    of the chain spaces of the ambient simplicial complex. This is possible using our definitions. 
}
	    \item
	    We give a definition
	    of barcodes for  semi-algebraic maps which are not necessarily proper (Definition~\ref{def:tilde-B}) generalizing the one for proper maps -- and we believe that this could form the basis of generalizing the results of the current papers to arbitrary semi-algebraic sets and maps. Similar ideas appear in 
	    \cite[Examples 15.11 and 15.14]{miller2017draft}, but our 
	    definition is adapted towards applications in real algebraic geometry.
	    
	    \item
	    By an application of a standard theorem in real algebraic geometry
	    (Hardt triviality theorem \cite{Hardt}) we can deduce that the topological  type  of the sub-level sets of a filtration of a semi-algebraic set  by a semi-algebraic function changes at only finitely many values of the function. This implies that the barcode of the original filtration is equal to that of a finite filtration (after proper definition of  barcodes encompassing both the finite and the continuous case as mentioned earlier). However, an algorithm based on Hardt triviality theorem would inevitably lead to a doubly exponential sized filtration -- since the proof of this theorem (see for example proof of \cite[Theorem 5.46]{BPRbook2}) depends on taking semi-algebraic triangulations for which only a doubly exponential complexity algorithm is known to exist. 
	    Another important contribution of the current paper is an algorithm with singly exponential complexity 
	    (see Algorithm~\ref{alg:filtration} below)
	    for reducing a given continuous filtration of a semi-algebraic set by a polynomial to a filtration of simplicial complexes indexed by a \emph{finite} subset of $\R$, such that the barcode of this finite filtration is equal to that of the continuous filtration in dimensions up to $\ell$. 
	    The two main ingredients for this algorithms are:
	    \begin{enumerate}[(a)]
	    \item
	    mathematical techniques introduced in \cite{BV06} for bounding the number of homotopy types of fibers of a semi-algebraic map;
	    \item
	    a recent algorithm for efficiently computing simplicial replacements of semi-algebraic sets \cite[Theorem 1]{basu-karisani}. 
	    \end{enumerate}
	    We note that Algorithm~\ref{alg:filtration} has other applications as well.
	    For example,  
	    it plays a key role in a recent work on computing a homology basis of the first  homology group of a given semi-algebraic set with singly exponential complexity \cite{Basu-Percival}.
	    
	    \item The last (and perhaps the most important) contribution is an
	    algorithm with a singly exponential complexity that computes the 
	    barcodes of a semi-algebraic filtration up to dimension $\ell$ for any fixed $\ell \geq 0$. After having reduced to the case of finite semi-algebraic filtration using Algorithm~\ref{alg:filtration},
	    we then 
	compute the barcode of this finite filtration of finite simplicial complexes (cf. Algorithms~\ref{alg:barcode-simplicial} and \ref{alg:barcode-semi-algebraic}) using Definition~\ref{def:barcode_multiplicity} and standard algorithms from linear algebra.\\
	
We remark that it is plausible that after ensuring the finiteness of the 
filtration, the last step of computing the barcode could be achieved by an appropriate extension of the algorithm for computing the first few Betti numbers of semi-algebraic sets described in \cite{Bas05-first}. However, this extension would be non-trivial and we prefer to use directly
Algorithm 3 in \cite{basu-karisani} for which no extension is needed.	\\

	    We prove the following theorem stated informally below. The formal statement appears later in the paper. 
	    
	\begin{theorem*}[cf. Theorem~\ref{thm:persistent}]
		There exists an algorithm(Algorithm~\ref{alg:barcode-semi-algebraic}) that takes as input
		a description of a closed and bounded semi-algebraic set $S \subset \R^k$, and a polynomial
		$P \in \R[X_1,\ldots,X_k]$, and outputs the ``barcodes''  (cf. Definition~\ref{def:barcode_multiplicity} below)
		in dimensions $0$ to $\ell$
		of the filtration of $S$ by the sub-level sets of the polynomial $P$. The complexity of
		this algorithm is bounded singly exponentially in $k$ (as a function of the number and degrees of polynomials appearing in the description of $S$).
	\end{theorem*}
The importance of the assumption that the input semi-algebraic subset be
closed and bounded is discussed in Section~\ref{subsubsec:non-proper}.
	\end{enumerate}

\subsection{Definition of complexity}
We will use the following notion of ``complexity''  in this paper. We follow the same definition as used in the book \cite{BPRbook2}. 
  
\begin{definition}[Complexity of algorithms]
\label{def:complexity}
In our algorithms we will usually take as input quantifier-free first order formulas whose terms
are  polynomials with coefficients belonging to an ordered domain $\D$ contained in a real closed field $\R$.
By \emph{complexity of an algorithm}  we will mean the number of arithmetic operations and comparisons in the domain $\D$.
If $\D = \mathbb{R}$, then
the complexity of our algorithm will agree with the  Blum-Shub-Smale notion of real number complexity \cite{BSS}.
In case, $\D = \Z$, then we are able to deduce the bit-complexity of our algorithms in terms of the bit-sizes of the coefficients
of the input polynomials, and this will agree with the classical (Turing) notion of complexity.
\end{definition}

\subsection{Prior and Related Work}
    As mentioned earlier,
    designing algorithms with singly exponential complexity for computing topological invariants of semi-algebraic sets has been at the center of research in algorithmic semi-algebraic geometry over the past decades. 
We refer the reader to the survey \cite{Basu-survey} 
for a history of these developments and contributions of many authors.
These algorithms are exact algorithms and work for all inputs. The complexity of an algorithm (see Definition~\ref{def:complexity})
is measured in terms of the number of arithmetic operations
in the ring $\D$ (and also in terms of the bit sizes if $\D = \Z$).
    
More recently, algorithms for computing Betti numbers of semi-algebraic sets have also been 
developed in other (more numerical) models of computations \cite{ BCL2019,BCT2020.1, BCT2020.2}.
In these papers the authors take a different approach. Working over $\mathbb{R}$, and given a 
well-conditioned semi-algebraic subset $S\subset \mathbb{R}^k$,  
they compute a witness complex whose geometric realization is $k$-equivalent to $S$. The size of this witness
complex is bounded singly exponentially in $k$. However, the complexity depends on the condition number of the input 
 (and so this bound is not uniform), and the algorithm will fail for ill-conditioned input when the condition number becomes
 infinite. This is unlike the kind of algorithms we consider in the current paper, which are supposed to work for all inputs
 and with uniform complexity upper bounds. 
 So these approaches are not comparable. 
 However, to the best of our knowledge there has not been any attempt to extend the
 numerical algorithms mentioned above for computing Betti numbers to computing
 persistent homology of semi-algebraic filtrations. \\

	The rest of the paper is organized as follows.
	In Section~\ref{sec:precise}, we give the precise statements of the main result after introducing the necessary definitions. 
	In Section~\ref{sec:results}, 
	we prove the key proposition (Proposition~\ref{prop:G2}) which allows us to efficiently reduce to the case of finite filtrations starting with a continuous one.
	In Section~\ref{sec:finite-algorithm},
	after introducing certain necessary preliminaries,
	we describe our algorithm for computing barcodes of semi-algebraic filtrations and analyze its complexity (thereby proving Theorem~\ref{thm:persistent}).
	Finally, in Section~\ref{sec:conclusion} we state some open questions and directions for future work in this area.

	\section{Precise definitions and statements of the main results}
	\label{sec:precise}
	In this section, 
	we define precisely persistent homology and
	barcodes of filtrations in Section~\ref{subsec:persistent}.
	Then in Section~\ref{subsec:filtration-sa} we define semi-algebraic filtrations and 
	state the main algorithmic result of the paper (Theorem~\ref{thm:persistent}).

	\subsection{Persistent homology and barcodes}
	\label{subsec:persistent}
	Let $T$ be an ordered set, and
	$\mathcal{F} = (X_t)_{t \in T}$, a tuple of subspaces of $X$, such that
	$s \leq t \Rightarrow X_s \subset X_t$. We call $\mathcal{F}$ a filtration
	of the  topological space $X$.
	
	We now recall the definition of the \emph{persistent homology
		groups}  associated to a filtration
	\cite{Edelsbrunner_survey, Weinberger_survey}.
 Since we only consider homology groups with rational coefficients, all homology groups in what follows are finite dimensional $\Q$-vector spaces.
	
	\begin{notation}
		\label{not:inclusion}
		For $s,t \in T, s \leq t$, and $p \geq 0$, we let $i_p^{s, t} : \HH_p (X_s) \longrightarrow
		\HH_p (X_t)$, denote the homomorphism induced by the inclusion $X_s
		\hookrightarrow X_t$.
	\end{notation}
	
	
	\begin{definition}
		\label{def:persistent}
		{\cite{Edelsbrunner_survey}} For each triple $(p, s, t)  \in \Z_{\geq 0} \times T \times T$ with 
		$
		s \leq t 
		$ 
		the   
		\emph{persistent homology  group}, $\HH_p^{s, t} (\mathcal{F})$ is defined by
		\begin{eqnarray*}
			\HH_p^{s, t} (\mathcal{F}) & = &  \mathrm{Im} (i_p^{s, t}).
		\end{eqnarray*}
		Note that $\HH_p^{s, t} (\mathcal{F}) \subset \HH_p (X_t)$, and $\HH_p^{s, s}
		(\mathcal{F}) = H_p (X_s)$.
	\end{definition}
	
	\begin{notation}
		We denote by $b_p^{s, t} (\mathcal{F}) = \dim_\Q (\HH_p^{s, t}
		(\mathcal{F}))$.
	\end{notation}

	Persistent homology measures how long a homology class persists in the filtration, in other words considering the homology classes as topological features, it gives an insight about the time  (thinking of the indexing set $T$ of the filtration as time) that a topological feature appears (or is born) and the time it disappears (or dies). This is made precise  as follows.  
	\begin{definition}
		For $s \leq t \in T$, and $ p \geq 0$,
		\begin{itemize}
			\item  we say that a homology class $\gamma \in \HH_p(X_{s})$ is \emph{born} at time $s$, if $\gamma \notin \HH_{p}^{s',s}(\mathcal{F})$, for any $s' < s$;
			\item for a class $\gamma \in \HH_p(X_{s})$ born at time $s$, we say that $\gamma$ \emph{dies} at time $t$,
			\begin{itemize}
				\item if  $i_p^{s,t'}(\gamma) \notin \HH_{p}^{s',t'}(\mathcal{F})$ for all $s', t'$ such that $s' < s \leq t' < t$,
				\item but $i_p^{s,t}(\gamma) \in \HH_{p}^{s'',t}(\mathcal{F})$, for some $s'' < s$.
			\end{itemize} 
		\end{itemize}  	
	\end{definition}
	
	\begin{remark}
		\label{rem:subspaces}
		Note that the homology 
		classes that are born at time $s$, and those that are born at time $s$ and dies at time $t$, 
		as defined above are not subspaces of $\HH_p(X_s)$. In order to be able to associate a ``multiplicity''
		to the set of homology classes which are born at time $s$ and dies at time $t$ we interpret them as classes in 
		certain \emph{subquotients}  of $\HH_*(X_s)$ in what follows.
	\end{remark}
	
	First observe that it follows from Definition~\ref{def:persistent} that for all $s' \leq s \leq t$ and $p \geq 0$, 
	$\HH_p^{s',t}(\mathcal{F})$ is a subspace of $\HH^{s,t}_p(\mathcal{F})$, and both are subspaces of
	$\HH_p(X_t)$. This is because the homomorphism $i^{s',t}_p = i^{s,t}_p\circ i^{s',s}_p$, and so the image of 
	$i^{s',t}_p$ is contained in the image of $i^{s,t}_p$.
	It follows that, for $s \leq t$,  the union of $\bigcup_{s' < s} \HH^{s',t}_p(\mathcal{F})$ is an increasing union of subspaces, and is 
	itself a subspace of $\HH_p(X_t)$.
	In particular, setting $t=s$,  $\bigcup_{s' < s} \HH^{s',s}(\mathcal{F})$ is a subspace of $\HH_p(X_s)$. \\
	
	With the same notation as  above:
	
	\begin{definition}[Subspaces of $\HH_p(X_s)$]
		\label{def:barcode_subspace}
		For $s \leq t$, and $p \geq 0$, we define
		\begin{eqnarray*}
			M^{s,t}_p(\mathcal{F}) &=& \bigcup_{s' < s} (i^{s,t}_p)^{-1}(\HH^{s',t}_p(\mathcal{F})), \\
			N^{s,t}_p(\mathcal{F}) &=& \bigcup_{s' < s\leq t' < t} (i^{s,t'}_p)^{-1}(\HH^{s',t'}_p(\mathcal{F})), \\
		\end{eqnarray*}
	\end{definition}

	\begin{remark}
		\label{rem:barcode_subspace_mean}
		The ``meaning''  of these subspaces are as follows.
		\begin{enumerate}[(a)] \label{rem:meaning} 
			\item
			For every fixed $s \in T$,  $M^{s,t}_p(\mathcal{F})$  is  a subspace of  $\HH_p(X_s)$ consisting of homology classes in $\HH_p(X_s)$ which are
			  \begin{center}
				``born before time $s$, or born at time $s$ and dies at $t$ or earlier'' 
			\end{center}
			
			\item
			Similarly,
			for every fixed $s \in T$,  $N^{s,t}_p(\mathcal{F})$  is  a subspace of $\HH_p(X_s)$ consisting of homology classes in  $\HH_p(X_s)$ which are
			\begin{center}
				``born before time $s$, or born at time $s$ and dies strictly earlier than $t$''
			\end{center}
		\end{enumerate}
		
		The dimensions of $M^{s,t}_p(\mathcal{F})$ and $N^{s,t}_p(\mathcal{F})$ are given in Eqn. \eqref{eq:dim_M} and \eqref{eq:dim_N} in Proposition~\ref{prop:finite:multiplicity} below.  
	\end{remark}
	
	We now define certain \emph{subquotients} of the homology groups of $\HH_p(X_s), s \in T, p \geq 0$,
	in terms of the subspaces defined above in Definition~\ref{def:barcode_subspace}. 
	
	\begin{definition}[Subquotients associated to a filtration]
		\label{def:barcode_subquotients}
		For $s \leq t$, and $p \geq 0$, we define
		\begin{eqnarray*}
			P^{s,t}_p(\mathcal{F}) &=& M^{s,t}_p(\mathcal{F})/N^{s,t}_p(\mathcal{F}), \\
			P^{s,\infty}_p(\mathcal{F}) &=&  \HH_p(X_s) / \bigcup_{s \leq t} M^{s,t}_p(\mathcal{F}).
		\end{eqnarray*}
		
		We will call 
		\begin{enumerate}[(a)]
			\item
			$P^{s,t}_p(\mathcal{F})$ the \emph{space of $p$-dimensional cycles  born at time $s$ and which dies at time $t$};
			and 
			\item
			$P^{s,\infty}_p(\mathcal{F})$ the \emph{space of $p$-dimensional cycles born at time $s$ and which never die}.
		\end{enumerate}
	\end{definition}
	
	\begin{remark}\label{rem:inclusion}
		Notice that 
		$M^{s,t}_p(\mathcal{F}) \subset M^{s,t'}_p(\mathcal{F})$ for $t \leq t'$, and hence
		$\bigcup_{s \leq t} M^{s,t}_p(\mathcal{F})$ is a subspace of $\HH_p(X_s)$,
		and  $N^{s,t}_p(\mathcal{F})$ is a subspace of 
		$M^{s,t}_p(\mathcal{F})$.
		Therefore, these subquotients are vector spaces and  
		have well defined dimensions. 
	\end{remark}

	Finally, we are able to achieve our goal of defining the multiplicity of a bar as the dimension of an associated vector space and define the barcode of a filtration.
	
	\begin{definition}[Persistent multiplicity, barcode]
		\label{def:barcode_multiplicity}
		We will denote for $s \in T,  t \in T \cup \{\infty \}$,
		\begin{equation}
			\label{eqn:def:barcode:multiplicity}
			\mu^{s,t}_p(\mathcal{F}) = \dim P^{s,t}_p(\mathcal{F}),
		\end{equation}
		and call $\mu^{s,t}_p(\mathcal{F})$ the \emph{persistent multiplicity of $p$-dimensional cycles born at time $s$ and dying at time $t$ if $t \neq \infty$,   or never dying in case $t = \infty$}.

		Finally, we will call the set
		\begin{equation}
		\label{eqn:def:barcode}
		\mathcal{B}_p(\mathcal{F}) = \{(s,t,\mu^{s,t}_p(\mathcal{F})) \mid \mu^{s,t}_p(\mathcal{F}) > 0\}
		\end{equation}
		\emph{the $p$-dimensional barcode associated to the filtration $\mathcal{F}$}. 
		
		We will call an element
		$b = (s,t,\mu^{s,t}_p(\mathcal{F})) \in \mathcal{B}_p(\mathcal{F})$ a \emph{bar of $\mathcal{F}$ of multiplicity
			$\mu^{s,t}_p(\mathcal{F}$}). 
	\end{definition}
	
	\begin{remark}
		\label{rem:barcode}
		We remark that the definition of multiplicity given up appears in an abstract setting in \cite[Corollary 7.3]{Crawley-Boevey2015}.
		Note also that the notion of \emph{persistent multiplicity} has been defined previously in 
		the context of finite filtrations (see \cite{Edelsbrunner-Harer}). The definition of
		$\mu^{s,t}_p(\mathcal{F})$ given in Eqn.~\eqref{eqn:def:barcode:multiplicity} generalizes
		that given in \emph{loc.cit.} in the  case of finite filtrations, who defined it
		using Eqn.~\eqref{eqn:prop:finite:multiplicity} in Proposition~\ref{prop:finite:multiplicity}
		stated below. Our definition gives a geometric meaning to this number as a dimension
		of a certain vector space (a subquotient of $\HH_p(X_s)$), and we prove that it agrees with that given in
		\emph{loc.cit.} in Proposition~\ref{prop:finite:multiplicity}.
		Also, it is important to note for what follows that our definition of a barcode applies uniformly to all filtrations with index coming from an ordered set, and we make no additional assumption on the
		indexing set.
	\end{remark}
	
	\begin{remark}[Continuous vs finite filtrations]
		\label{rem:continuous}
		In most applications the filtration $\mathcal{F}$ is assumed to be finite (i.e. the ordered set $T$ is finite). Since we are considering filtration of semi-algebraic sets by the sub-level  sets of a polynomial
		function, our filtration is indexed by $\R$ and is an example of a continuous (infinite) filtration.
		Nevertheless, we will reduce to the finite filtration case by proving that the barcode of the given
		filtration 
		is equal to that of a finite filtration. A general theory encompassing both finite and 
		infinite filtrations using a categorical view-point has been developed (see \cite{Bubenik-Scott, Oudot}). We avoid
		using the categorical definitions and the module-theoretic language 
		used in \cite{Oudot}. We will prove directly the equality of the barcodes of the infinite and the 
		corresponding finite filtration(cf. Proposition~\ref{prop:G2}) that is important in designing our algorithm,
		starting from the definition of persistent multiplicities given above.
	\end{remark}
	
	We now give a concrete example of a barcode associated to a 
	(infinite) filtration.
	
	\begin{example}
		\label{eg:torus}
		Let $S$ be the two-dimensional torus 
		(topologically $\Sphere^1 \times \Sphere ^1$)
		embedded in $\mathbb{R}^3$, and $\mathcal{F}$ be
		the filtration of the torus by the sub-level sets of the height function  (depicted in Figure~\ref{Fig:tor}). 
		We denote by $S_{\leq t}$ the subset of the torus having ``height'' $\leq t$.
		
		We consider homology in dimensions  $0$, $1$ and $2$. 
		
		Informally, one observes that a $0$-dimensional homology class is born at time $t_0$ which never dies. There are two $1$-dimensional homology classes, the horizontal loop born at time $t_2$ and the vertical loop born at time $t_4$, which also never die. Lastly, there is a $2$-dimensional homology class born at time $t_5$ which never dies. Since there are no homology classes of the same dimension being born and dying at the same time, multiplicities in all the cases are 1. 
		
		More formally, following Definitions~\ref{def:barcode_subspace}, \ref{def:barcode_subquotients} and \ref{def:barcode_multiplicity}, we obtain:
		
		\begin{enumerate}[({Case p = }1)]
			\setcounter{enumi}{-1}
			\item If $t_0 \leq t < \infty$ then
			(using Definition~\ref{def:barcode_subspace}) 
			\[
			M^{t_0,t}_0(\mathcal{F}) = 0,
			\]
			and  hence
			(using Definitions~\ref{def:barcode_subquotients} and \ref{def:barcode_multiplicity})
			\[
			P^{t_0,t}_0(\mathcal{F}) = 0, \mbox{ and }
			\mu^{t_0, t}_0(\mathcal{F})=0.
			\]
			On the other hand,  
			\[
			P^{t_0,\infty}_0(\mathcal{F}) = \HH_0(S_{\leq t_0}),
			\]
			implying 
			\[
			\mu^{t_0, \infty}_0(\mathcal{F})=1.
			\]
			
			\item  For $t_2 \leq t < \infty$, 
			\[
			M^{t_2,t}_1(\mathcal{F}) = 0,
			\]
			and hence  
			\[P^{t_2,t}_1(\mathcal{F}) = 0, \mbox{ and }
			\mu^{t_2, t}_1(\mathcal{F})=0.
			\]
			Moreover,  
			\[
			P^{t_2,\infty}_1(\mathcal{F}) = \HH_1(S_{\leq t_2}),
			\]
			and therefore, 
			\[
			\mu^{t_2, \infty}_1(\mathcal{F})=1.
			\]
			
			\noindent For $t_4 \leq t < \infty$, 
			\[
			M^{t_4,t}_1(\mathcal{F}) = N^{t_4,t}_1(\mathcal{F}) =  \HH_1(S_{< t_4}),
			\]
			and hence 
			\[
			P^{t_4,t}_1(\mathcal{F}) = 0, \mbox{ and }
			\mu^{t_4, t}_1(\mathcal{F})=0.
			\]
			Moreover,  
			\[
			P^{t_4,\infty}_1(\mathcal{F}) =  \HH_1(S_{\leq t_4}) /  \HH_1(S_{< t_4}),
			\]
			and therefore  
			\[
			\mu^{t_4, \infty}_1(\mathcal{F})=1.
			\]
			
			\item For $t_5 \leq t < \infty$, 
			\[
			M^{t_5,t}_2(\mathcal{F}) = 0,
			\]
			and hence 
			\[
			P^{t_5,t}_2(\mathcal{F}) = 0, \mbox{ and }
			\mu^{t_5, t}_2(\mathcal{F})=0.
			\]
			Moreover, 
			\[
			P^{t_5,\infty}_2(\mathcal{F}) =\HH_2(S),
			\]
			and therefore 
			\[
			\mu^{t_5, \infty}_2(\mathcal{F})=1.
			\]
		\end{enumerate}
		  Therefore the barcodes are as follows 
		  (using Eqn.~\eqref{eqn:def:barcode}).
		
		$$ {\small
			\begin{array}{ccl}
				\mathcal{B}_0(\mathcal{F}) & = & \{(t_0, +\infty, 1)\}, \\
				\mathcal{B}_1(\mathcal{F}) & = & \{(t_2, +\infty, 1), (t_4, +\infty, 1)\},\\
				\mathcal{B}_2(\mathcal{F}) & = &\{(t_5, +\infty, 1)\}.
		\end{array}}
		$$ 
		Figure~\ref{Fig:bar} illustrates the corresponding bars.  Notice that even though the filtration $\mathcal{F}$ is an infinite filtration indexed by $\mathbb{R}$, the barcodes, $\mathcal{B}_p(\mathcal{F})$, are finite.
		
		\begin{figure}[h!]
			\centering
			\subfigure[]{%
				\label{Fig:tor}%
				\includegraphics[height=1.12in]{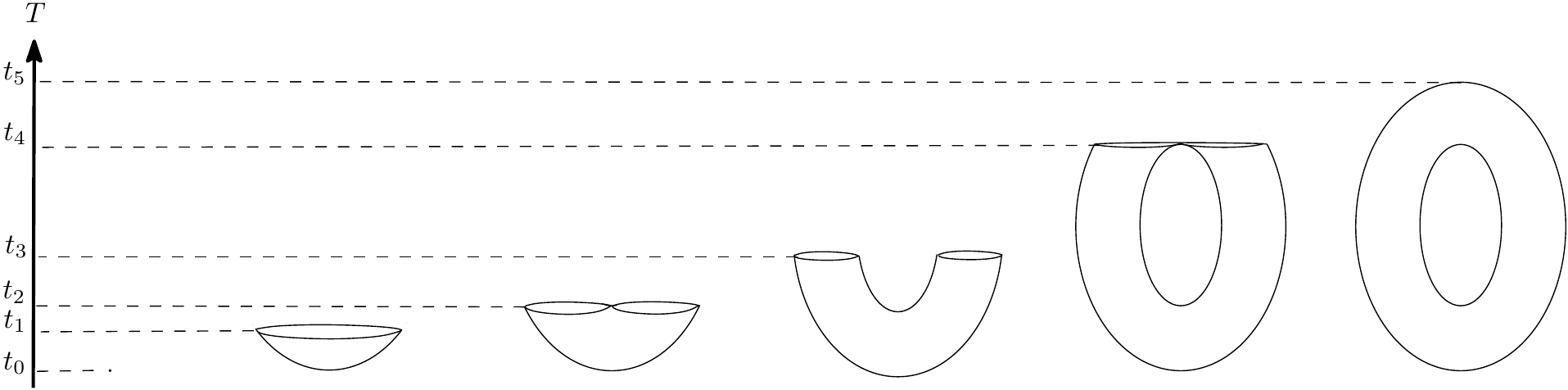}}%
			\qquad
			\subfigure[]{%
				\label{Fig:bar}%
				\includegraphics[height=1in]{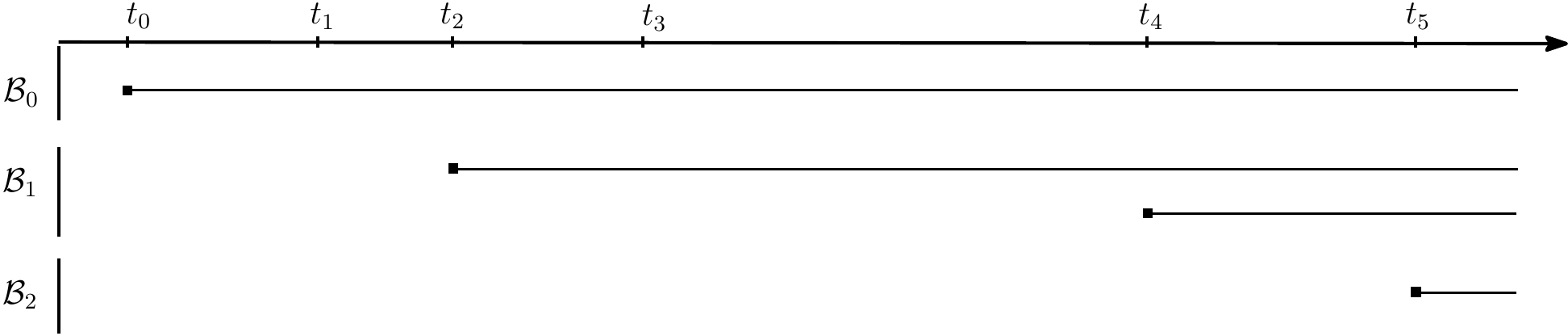}}%
			
			\caption{\small (a) Torus filtered by the sub-level sets of the height function, (b) corresponding barcodes for homology classes of dimension 0, 1 and 2. }
		\end{figure}
	\end{example}
	
	The main type of filtration that we consider in this paper is filtration of semi-algebraic sets by the sub-level sets of continuous semi-algebraic functions -- which we define below.
	
	\subsubsection{$\mathcal{P}$-formulas and $\mathcal{P}$-semi-algebraic sets}
	\begin{notation}[Realizations, $\mathcal{P}$-, $\mathcal{P}$-closed
		semi-algebraic sets]
		\label{not:sign-condition} 
		For any finite set of polynomials $\mathcal{P}
		\subset \R [ X_{1} , \ldots ,X_{k} ]$, 
		we call any quantifier-free first order formula $\phi$ with atoms, $P =0, P < 0, P>0, P \in \mathcal{P}$, to
		be a \emph{$\mathcal{P}$-formula}. 
		Given any semi-algebraic subset $Z \subset \R^k$,
		we call the realization of $\phi$ in $Z$,
		namely the semi-algebraic set
		\begin{eqnarray*}
			\RR(\phi,Z) & := & \{ \mathbf{x} \in Z \mid
			\phi (\mathbf{x})\}
		\end{eqnarray*}
		a \emph{$\mathcal{P}$-semi-algebraic subset of $Z$}.

		We say that a quantifier-free formula $\phi$ is \emph{closed}  
		if it is a formula in disjunctive normal form with no negations, and with atoms of the form $P \geq 0, P \leq 0$ (resp. $P > 0, P < 0$),  
		where $P \in \D[X_1,\ldots,X_k]$. If the set of polynomials appearing in a closed (resp. open) formula
		is contained in a finite set $\mathcal{P}$, we will call such a formula a $\mathcal{P}$-closed 
		formula,
		and we call
		the realization, $\RR \left(\phi \right)$, a \emph{$\mathcal{P}$-closed 
			semi-algebraic set}.
	\end{notation}

	\subsection{Semi-algebraic filtrations}
	\label{subsec:filtration-sa}
	
	We consider the algorithmic problem of computing
	the dimensions of persistent homology groups and barcodes of the filtration induced on a given semi-algebraic set by a polynomial function.  
	
	\begin{definition}
		\label{def:filtration-sa}
		Let $S \subset \R^k$ be a 
		semi-algebraic set and 
		$P:S \rightarrow \R$ a continuous semi-algebraic map.
		
		For $t \in \R \cup \{\pm \infty\}$, let
		\[
		S_{P \leq t} = \{ x \in S \mid P(x) \leq t \}.
		\]
		Then, $(S_{P \leq t})_{t\in \R \cup \{\pm \infty\}}$ is a filtration of the semi-algebraic set $S$ indexed by $\R \cup \{\pm \infty\} $, and we
		will denote this filtration by $\mathcal{F}(S,P)$.
	\end{definition}
	
	\begin{notation}
	  For $p \geq 0$, we will denote  
	  \[
	  \mathcal{B}_p(S,P) = \mathcal{B}_p(\mathcal{F}(S,P)).
	  \]
    \end{notation}

	  \begin{remark}
	  \label{rem:homology}
	  In the definition of $\mathcal{B}_p(\mathcal{F}(S,P))$ we need to specify the homology theory we are using.
	  For a semi-algebraic set $X$ defined over an arbitrary real closed field $\R$ we take homology groups $\HH_*(X) = \HH_*(X,\Q)$ as defined in \cite[(3.6), page 141]{Delfs-Knebusch-Crelle}. It agrees with singular homology in case $\R = \mathbb{R}$.
	  \end{remark}
	
	\begin{remark} \label{cech-example} 
	Note that many filtrations commonly used in computational topology are examples of filtrations of
	semi-algebraic sets by polynomial functions as defined above. One example of this is the well-known \emph{\v{C}ech-complex} \cite{dey2022computational} which can be described as follows.
	
	Let $\{\x^{(1)},\ldots,\x^{(n)}\}$ be a finite set of points in $\R^k$, and let
	$S \subset \R^{k+1}$, be the semi-algebraic set defined by the formula
	\[
	\phi(X_1,\ldots,X_k,T):= \bigvee_{i=1}^n \left(|\X - \x^{(i)}|^2 - T \leq 0\right),
	\]
	where $\X = (X_1,\ldots,X_k)$.
	Let $P = T$. Then, the filtration $\mathcal{F}(S,P)$ is homeomorphic to the filtration obtained by
	taking unions of balls of growing radius centered at  $\{\x^{(1)},\ldots,\x^{(n)} \}$.
	This latter filtration plays a very important role in applications (for example, in analyzing the topological
	structure of point-cloud data).
	\end{remark}
	
	\begin{remark}
	Note also that the barcode of a polynomial function restricted to a 
	semi-algebraic set $S$ gives important
	topological information about the function $P$ on $S$. It allows one to define a $p$-dimensional distance between
	two such polynomial functions restricted to $S$, by defining a notion of distance between two barcodes.
	Various distances have been proposed but the most commonly used one is the so called ``bottle-neck distance'' \cite{Edelsbrunner-Harer}. 
	An  algorithm  with singly exponential complexity for computing the barcode of a polynomial also gives an algorithm with singly exponential complexity  
	for computing such distances as well. To our knowledge the algorithmic problem of computing barcodes of 
	polynomial functions on semi-algebraic sets have not been considered prior to our work.
\end{remark}	
	
	We prove the following theorem.
	
	\begin{theorem}
		\label{thm:persistent}
		There exists an algorithm that takes as input:
		\begin{enumerate}[1.]
			\item a finite set of polynomials, $\mathcal{P} \subset \D[X_1,\ldots,X_k]$;
			\item
			a $\mathcal{P}$-closed formula $\phi$ 
			such that $\RR(\phi)$ is bounded;
			\item
			a polynomial $P \subset \D[X_1,\ldots,X_k]$;
			\item $\ell \geq 0$;
		\end{enumerate}
		and computes $\mathcal{B}_p(\RR(\phi),P)$, for $0 \leq p \leq \ell$.
		The complexity of the algorithm is bounded by $(s d)^{k^{O(\ell)}}$, where $s = \card(\mathcal{P})$,
		and $d$ is the maximum amongst the degrees of $P$ and the polynomials in $\mathcal{P}$.
	\end{theorem}
	
	\subsubsection{Barcodes of non-proper maps}
	\label{subsubsec:non-proper}
	Notice that in Theorem~\ref{thm:persistent} we only consider semi-algebraic sets $S$ which are closed and bounded. In particular, this implies that any continuous semi-algebraic function on $S$ is a proper map $S \rightarrow \R$ (i.e. the inverse image of a closed and bounded semi-algebraic set is closed and bounded). 
	
	One reason to assume the properness is
	that for non-proper semi-algebraic maps $P:S \rightarrow \R$, the barcode $\mathcal{B}_p(S,P)$ may not reflect the topology of $S$ as illustrated in the following example 
	(see also \cite[Examples 15.11 and 15.14]{miller2017draft}).
	
	\begin{example}
	   \label{eg:non-proper}
	   Let $S \subset \R^2$ be the (unbounded) semi-algebraic set defined by the formula
	 \[
	 \phi := (0 < X_1 < 1) \wedge (X_1(X_1 -1)X_2 -1 = 0)
	 \]
	 (depicted in Figure~\ref{fig:diagram}), and let $P = X_1$.
	 Consider the semi-algebraic filtration $\mathcal{F}(S,P)$.
	 Note that $P$ restricted to $S$ is not a proper semi-algebraic map 
	 ($P^{-1}([0,1])$ is not bounded).

        \begin{figure}[h]
		\centering
		\includegraphics[scale=.4]{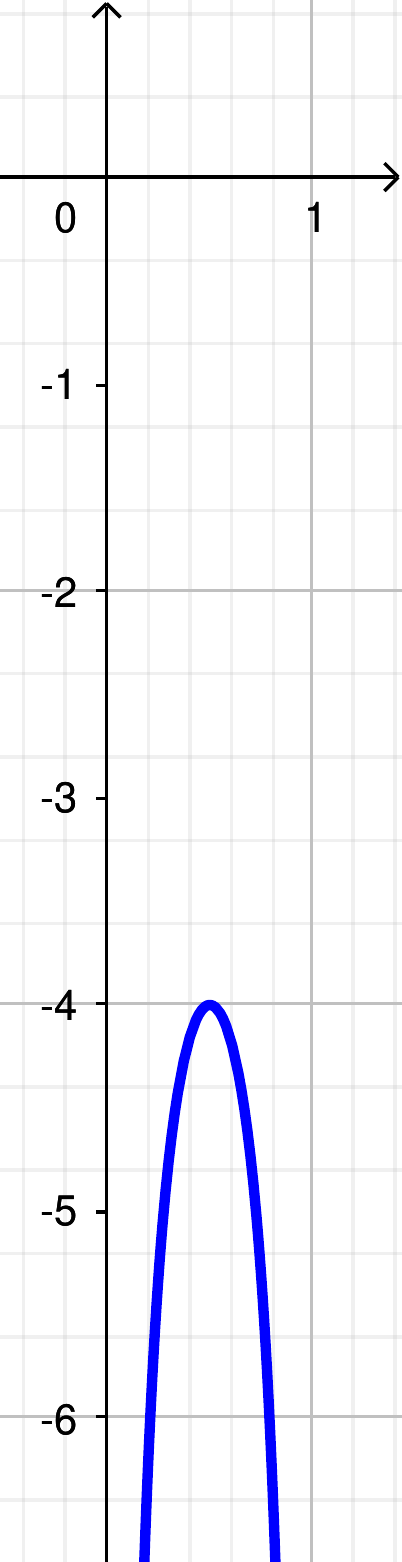}%
		\caption{\small $S = \{ (x_1,x_2) \  | \  0 < x_1 < 1, x_1(x_1-1)x_2 - 1 = 0 \}$}
		\label{fig:diagram}
		\end{figure}
		
	 It is clear that for $p > 0$, 
	 \[
	 \mathcal{B}_p(S,P) = \emptyset.
	 \]
	 We claim that even for $p=0$ (contrary to the expectation)
	 \[
	 \mathcal{B}_p(S,P) = \emptyset.
	 \]
	 To see this observe that 
	 for all $s \leq 0$, $t \geq s$, we have that
	 \begin{eqnarray*}
	 M^{s,t}_0(\mathcal{F}(S,P)) &=& \bigcup_{s' < s} (i^{s,t}_p)^{-1}(\HH^{s',t}_0(\mathcal{S})) \\
	 &=& 0,\\
	 \HH_0(S_{P \leq s}) &=& 0,
	\end{eqnarray*}
	since $S_{P \leq s} = \emptyset$ for $s \leq 0$.
	This shows that 
	\begin{equation}
	\label{eqn:eg:non-proper:1}
	\mu^{s,t}_0(\mathcal{F}(S,P)) = 0,  s \leq 0, t \geq s.   
	\end{equation}

	For $s >0$, and $t \geq s$, it follows from Definition~\ref{def:barcode_multiplicity} that
	\[
	M^{s,t}_0(\mathcal{F}(S,P)) = N^{s,t}_0(\mathcal{F}(S,P)) = \HH_0(S_{P \leq s}),
	\]
	proving that
	\begin{equation}
	\label{eqn:eg:non-proper:2}
	 \mu^{s,t}_0(\mathcal{F}(S,P)) = 0, s > 0, t \geq s.   
	\end{equation}

	Together Eqns. \eqref{eqn:eg:non-proper:1} and \eqref{eqn:eg:non-proper:2} imply that
	\[
	 \mathcal{B}_0(S,P) = \emptyset.
	 \]
\end{example}

	 In order to have a more reasonable definition of barcodes (and allow
	 ``bars'' which have open endpoints) we propose the following definition. We use two notions from real algebraic geometry -- that of the real spectrum and the real closed extension of $\R$ by the field of Puiseux series.
	 
	 Let $S \subset \R^k$ be an arbitrary semi-algebraic set and 
	 $P:S \rightarrow \R$ a continuous semi-algebraic function.
	 We define a new filtration $\widetilde{\mathcal{F}}(S,P)$ as follows.
	 
	 The indexing set of the new filtration will the set
	 \[
	 \widetilde{\R} = \{-\infty, +\infty\} \cup \bigcup_{x \in \R} \{x_-,x,x_+\},
	 \]
	 on which a total order is  specified by 
	 \[
	 -\infty < x_- < x < x_+ < y_- < y < y_+ < \infty,
	 \]
	 for all $x < y$ in $\R$.
	 (The ordered set $\widetilde{\R}$ is the \emph{real spectrum} of the ring $\R[X]$ -- see for example \cite[page 134]{BCR}).
	 
	 We now define the filtration $\widetilde{\mathcal{F}}(S,P)$.
	 \begin{definition}[Filtration for semi-algebraic maps not necessarily proper]
	 \label{def:tilde-F}
	 For  $\tilde{t} \in \widetilde{\R}$ define
	\begin{eqnarray*}
	    \widetilde{S}_{\tilde{t}} &=& \E(S,\R\la\eps\ra)_{P = -1/\eps}, \mbox{ if $\tilde{t}=-\infty$},\\
		 &=& \E(S,\R\la\eps\ra)_{P  \leq t-\eps}, \mbox{ if $\tilde{t} = t_-, t \in \R$},\\
		 &=& \E(S,\R\la\eps\ra)_{P  \leq t}, \mbox{ if $\tilde{t} = t \in \R$},\\
		 &=& \E(S,\R\la\eps\ra)_{P  \leq t+\eps}, \mbox{ if $\tilde{t} = t_+, t \in \R$},\\
		 &=& \E(S,\R\la\eps\ra)_{P = 1/\eps}, \mbox{ if $\tilde{t}=+\infty$}.
	\end{eqnarray*}
	(where $\R\la\eps\ra$ is the field of algebraic Puisuex series in
	$\eps$, and $\E(\cdot,\R\la\eps\ra)$ denotes the extension of a semi-algebraic subset of $\R^k$ to $\R\la\eps\ra^k$ -- see Notation~\ref{not:Puiseux} and Notation~\ref{not:extension}).
	\end{definition}
	
	\begin{definition}[Barcode for filtration induced by a semi-algebraic map not necessarily proper]
	    \label{def:tilde-B}
	    For $S \subset \R^k$  an arbitrary semi-algebraic set and 
	 $P:S \rightarrow \R$ a continuous semi-algebraic function,
	 we define
	\[
	\widetilde{\mathcal{B}}_p(S,P) = \mathcal{B}_p(\widetilde{\mathcal{F}}(S,P)).
	\]
	\end{definition}

	It is easy now to verify that for the pair $S,P$ in Example~\ref{eg:non-proper} 
	\[
	\widetilde{\mathcal{B}}_0(S,P) = \{(0_+,+\infty,1)\}.
	\]

Note that 
	\[
	\widetilde{\mathcal{B}}_p(S,P)
	 \subset \widetilde{\R} \times \widetilde{\R} \times \Z_{>0}.
	\]    
Using Hardt triviality theorem, one can deduce that 
$\widetilde{\mathcal{B}}_p(S,P)$ is a finite set.
We will formally prove this statement later for proper semi-algebraic maps
(see Proposition~\ref{prop:finite}).

The barcode for a proper semi-algebraic map takes its value in
$\R \times \R \times \Z_{>0}$ which is properly contained in the 
$\widetilde{\R} \times \widetilde{\R} \times \Z_{>0}$.
It is not difficult to prove that in case $P:S \rightarrow \R$
is a proper semi-algebraic map, the new definition of barcode
agrees with the previous one. 

We record the above mentioned facts in the following proposition for future reference and omit the proofs. 
We will not use it 
in this paper since we restrict ourselves to the proper case.

\begin{proposition}
\label{prop:proper}
For any continuous semi-algebraic map $P:S \rightarrow \R$ and for all
$p \geq 0$, $\widetilde{B}(S,P)$ is a finite set. Moreover,
if $P$ is a proper semi-algebraic map, then for all
$p \geq 0$,
	\[
	\widetilde{\mathcal{B}}_p(S,P) = \mathcal{B}_p(S,P).
	\]
\end{proposition}
	 \begin{proof}
	     Omitted.
	 \end{proof}

	\section{Continuous to finite filtration}
	\label{sec:results}
	In this section we describe how to efficiently reduce the problem
	of computing the barcode of a continuous semi-algebraic filtration to that of a finite filtration of semi-algebraic sets.
	The mathematical results are encapsulated in Propositions~\ref{prop:finite} and \ref{prop:G2} stated and proved in 
	Section~\ref{subsec:filtration}. Then in Section~\ref{subsec:multiplicities} we prove a formula used to compute the  barcode of a finite filtration (Proposition~\ref{prop:finite:multiplicity}). This formula is not new (see  \cite[page 152]{Edelsbrunner-Harer}\cite{Edelsbrunner-Harer}),  
	however, it is important to deduce that from our new definition of barcodes.  
	
	Recall that we are interested in the persistent homology of filtrations of semi-algebraic sets by the sub-level sets of a polynomial. Recall also (cf. Definition~\ref{def:filtration-sa})  that for  a closed and bounded semi-algebraic set $S \subset \R^k$,  $P \in \R[X_1,\dots, X_k ]$,
	and $t \in \R \cup \{\pm \infty\}$, we denote the filtration
	\[
	\left(S_{P \leq t} = \{ x \in S \mid P(x) \leq t \}\right)_{t\in \R \cup \{\pm \infty\}}
	\]
	by $\mathcal{F}(S,P)$.
	
	Our first observation is that,
	even though the indexing set $\R \cup \{\pm \infty \}$ is infinite, 
	for each $p \geq 0$, the barcode
	$\mathcal{B}_p(\mathcal{F}(S,P))$ is a finite set (cf. Example \ref{eg:torus}).

	\begin{proposition}
		\label{prop:finite}
		For each $p \geq 0$,
		the cardinality of  $\mathcal{B}_p(\mathcal{F}(S,P))$ is finite.
	\end{proposition}

	\subsection{Reduction to the case of a finite filtration}
	\label{subsec:filtration}
	We will now prove a result (cf. Proposition~\ref{prop:G2} below)  
	from which Proposition~\ref{prop:finite} will follow. 
	Our strategy is to identify a finite set of values $\{s_0,\ldots,s_M\} \subset \R$, such that
	the semi-algebraic homotopy type of the increasing family $S_{P \leq t}$ (as $t$ goes from $-\infty$ to $\infty$), can change
	only when  $t$ crosses one of the $s_i$'s. This would imply that the barcode, $\mathcal{B}_p(\mathcal{F}(S,P))$, of the infinite filtration $\mathcal{F}(S,P)$,
	is \emph{equal}  to the barcode of the \emph{finite filtration}
	$\emptyset \subset S_{P \leq s_0} \subset \cdots \subset S_{P \leq s_M} \subset S$ (cf. Proposition~\ref{prop:G2} below). In addition, we will obtain
	a bound on the number $M$ in terms of the number of polynomials appearing in the definition
	of $S$ and their degrees, as well as the degree of the polynomial $P$. The technique used in the
	proofs of these results are adaptations of the technique used in the proof of the main result  (Theorem 2.1)
	in \cite{BV06}, which gives a singly exponential
	bound on the number of distinct homotopy types amongst the fibers of a semi-algebraic map
	in \cite{BV06}. We need a slightly different statement than that of Theorem 2.1 in \cite{BV06}. However,
	our situation is simpler since we only need the result for maps to $\R$ (rather than to $\R^n$ as is the case in \cite[Theorem 2.1]{BV06}).
	
	\subsubsection{Real closed extensions and Puiseux series}
	We will need some
	properties of Puiseux series with coefficients in a real closed field. We
	refer the reader to \cite{BPRbook2} for further details.
	
	\begin{notation}
	\label{not:Puiseux}
		For $\R$ a real closed field we denote by $\R \left\langle \eps
		\right\rangle$ the real closed field of algebraic Puiseux series in $\eps$
		with coefficients in $\R$. We use the notation $\R \left\langle \eps_{1},
		\ldots, \eps_{m} \right\rangle$ to denote the real closed field $\R
		\left\langle \eps_{1} \right\rangle \left\langle \eps_{2} \right\rangle
		\cdots \left\langle \eps_{m} \right\rangle$. Note that in the unique
		ordering of the field $\R \left\langle \eps_{1}, \ldots, \eps_{m}
		\right\rangle$, $0< \eps_{m} \ll \eps_{m-1} \ll \cdots \ll \eps_{1} \ll 1$.
		
	\end{notation}
	
	\begin{notation}
		\label{not:lim}
		For elements $x \in \R \left\langle \eps \right\rangle$ which are bounded
		over $\R$ we denote by $\lim_{\eps}  x$ to be the image in $\R$ under the
		usual map that sets $\eps$ to $0$ in the Puiseux series $x$.
	\end{notation}
	
	\begin{notation}
		\label{not:extension}
		If $\R'$ is a real closed extension of a real closed field $\R$, and $S
		\subset \R^{k}$ is a semi-algebraic set defined by a first-order formula
		with coefficients in $\R$, then we will denote by $\E(S, \R') \subset \R'^{k}$ the semi-algebraic subset of $\R'^{k}$ defined by
		the same formula.
		It is well known that $\E(S, \R')$ does
		not depend on the choice of the formula defining $S$ 
		\cite[Proposition 2.87]{BPRbook2}.
	\end{notation}
	
	\begin{notation}
		\label{not:monotone}
		Suppose $\R$ is a real closed field,
		and let $X \subset \R^k$ be a closed and bounded  semi-algebraic subset, and $X^+ \subset \R\la\eps\ra^k$
		be a semi-algebraic subset bounded over $\R$. 
		Let for $t \in \R, t >0$, $\widetilde{X}^+_{t} \subset \R^k$ denote the semi-algebraic
		subset obtained by replacing $\eps$ in the formula defining $X^+$ by $t$, and it is
		clear that  for $0 < t \ll 1$, $\widetilde{X}^+_t$ does not depend on the formula chosen. We say that $X^+$ is \emph{monotonically decreasing to $X$}, and denote $X^+ \searrow X$ if the following conditions are satisfied.
		\begin{enumerate}[(a)]
			\item
			for all $0 < t < t'  \ll 1$,  $\widetilde{X}^+_{t} \subset  \widetilde{X}^+_{t'}$;
			\item
			\[
			\bigcap_{t > 0} \widetilde{X}^+_{t} = X;
			\]
			or equivalently $\lim_\eps X^+ =  X$.
		\end{enumerate}
		More generally,
		if $X \subset \R^k$ be a closed and bounded  semi-algebraic subset, and $X^+ \subset \R\la\eps_1,\ldots,\eps_m\ra^k$
		a semi-algebraic subset bounded over $\R$,
		we will say $X^+ \searrow X$ if and only if 
		\[
		X^+_{m+1} = X^+ \searrow X^+_m,  \;X^+_m \searrow X^+_{m-1}, \ldots, X^+_{2} \searrow X^+_1 = X,
		\]
		where for $i=1,\ldots, m$, $X^+_i = \lim_{\eps_i} X^+_{i+1}$. 
		\end{notation}

	The following lemma will be useful later.
	\begin{lemma}
		\label{lem:monotone}
		Let $X \subset \R^k$ be a closed and bounded  semi-algebraic subset, and $X^+ \subset \R\la\bar\eps_1,\ldots,\bar\eps_m\ra^k$
		a semi-algebraic subset bounded over $\R$, such that
		$X^+ \searrow X$.
		Then,
		$\E(X,\R\la\bar\eps_1,\ldots,\bar\eps_m\ra)$ is semi-algebraic deformation retract of $X^+$.
	\end{lemma}
	
	\begin{proof}
		See proof of Lemma 16.17 in 
		\cite{BPRbook2}.
	\end{proof}
	
	\subsubsection{Outline of the reduction}
	\label{subsubsec:finite-outline}
	Before delving into the detail we first give an outline of the main idea behind the reduction to the finite filtration case.
	The key mathematical result that we need is the following. Given a  semi-algebraic subset $X \subset \R^{k+1}$, obtain a semi-algebraic partition of $\R \cup \{\pm \infty\}$ into points $-\infty = s_{-1} < s_0 < s_1 < \cdots < s_M < s_{M+1} = \infty$, and open 
	intervals $(s_i,s_{i+1}), -1 \leq i \leq M$, such that the homotopy type of $X_t = X \cap \pi_{k+1}^{-1}$ stays constant over each open interval $(s_i,s_{i+1})$ (here $\pi_{k+1}$ denotes the projection
	on the last coordinate). In our application the fibers $X_t$ will be a non-decreasing in $t$ (in fact, $X_t$ will be equal to $S_{P \leq t}$) but we do not need this property to hold for obtaining the 
	partition mentioned above.
	
	The following example  is illustrative. 
	
	\begin{figure}[h]
		\centering
		\includegraphics[scale=1.00]{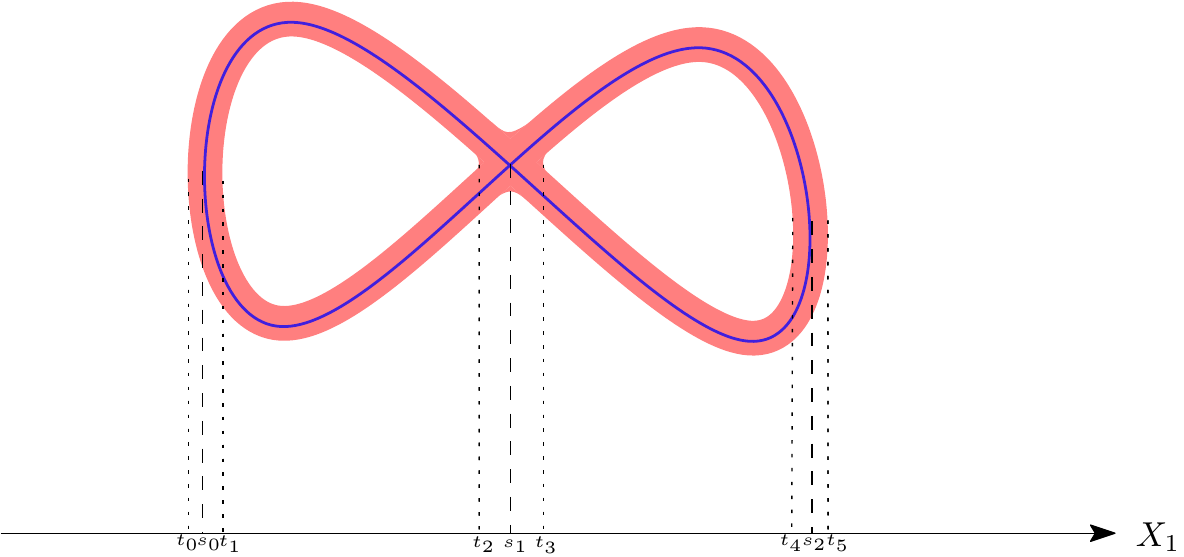}%
		\caption{\small Homotopy types of fibers}
		\label{fig:eight}
	\end{figure}

	Suppose that $X \subset \R^{2}$ is a singular curve shown in blue in Figure~\ref{fig:eight}. We define a semi-algebraic
	tubular neighborhood $X^\star(\eps)$ of $X$ using an infinitesimal $\eps$ (shown in red), 
	whose boundary has good algebraic properties -- namely, in this case a finite number of critical values $t_0 < t_1 < \cdots < t_5$ for the projection map onto the chosen coordinate $X_{k+1}$ which
	is shown as $X_1$ in the figure. The $t_i$'s give a partition of $\R\la\eps\ra$ rather than that of
	$\R$, 
	and over each interval $(t_i,t_{i+1})$ the semi-algebraic  homeomorphism type of $X^\star(\eps)$ (but not necessarily the semi-algebraic homotopy type of $\E(X,\R\la\eps\ra)$) stay constant.
	Clearly this partition does not have the homotopy invariance property with respect to the
	set $\E(X,\R\la\eps\ra)$. However, the intervals $(t_1,t_2) \cap \R = (s_0,s_1)$ and $(t_3,t_4) \cap \R = (s_1,s_2)$ does have the require property with respect to $X$, and the points $s_0,s_1,s_2$
	gives us the require partition.
	
	In the general case the definition of the tube $X^\star(\eps)$ is more involved and uses more than
	one infinitesimal (cf. Notation~\ref{not:start}). 
	The set of points corresponding to the $t_i$'s  in the above example is defined precisely in 
	Proposition~\ref{prop:G}  where the  important property of the partition of $\R\la\bar\eps\ra$ they 
	induce is also proved. The passage from the $t_i$'s to the $s_i$'s 
	and the important property satisfied by the $s_i$'s is described in Lemma~\ref{lem:removal}.
	The finite set of values $\{s_0,\ldots,s_M\} \subset \R$ is then used to define a finite filtration
	of the given semi-algebraic set, and the fact that this finite filtration has the same barcode as the 
	infinite filtration we started with is proved in Proposition~\ref{prop:G2}.
	Proposition~\ref{prop:G2} immediately implies Proposition~\ref{prop:finite}.
	
	There are several further technicalities involved in converting the above construction into an efficient algorithm. These are explained in Section~\ref{sec:finite-algorithm}. The complexity of the whole
	procedure is bounded singly exponentially.
	
	\subsubsection{Proof of Proposition~\ref{prop:finite}}
	\label{subsubsec:finite-prop-2}
	We begin by fixing some notation.
	
	\begin{notation}
		\label{not:Z}
		For $\mathcal{Q} \subset \R[X_1,\ldots,X_k]$ we will denote by 
		\[
		\ZZ(\mathcal{Q},\R^k) = \{ x \in \R^k | \bigwedge_{Q \in \mathcal{Q}} Q(x) = 0\}.
		\]
		For $Q \in \R[X_1,\ldots,X_k]$, we will denote by $\ZZ(Q,\R^k) = \{x \in \R^k \mid Q(x) = 0\}$.
	\end{notation}
	
	\begin{definition}
		\label{def:sign-condition}
		Let $\mathcal{Q}$ be a finite subset of $\R [X_{1} , \ldots ,X_{k} ]$. 
		A \emph{sign condition}  on $\mathcal{Q}$ is an element of $\{0,1, -1\}^{\mathcal{Q}}$.
		We say that $\mathcal{Q}$
		\emph{realizes} the sign condition $\sigma$ at $x \in \R^{k}$
		if
		\[
		\bigwedge_{Q \in \mathcal{Q}} \sign(Q(x)) = \sigma(Q).
		\]
		
		The \emph{realization of the sign condition
			$\sigma$} is 
		\[ \RR ( \sigma ) = \{x \in \R^{k}   \mid  
		\bigwedge_{Q \in \mathcal{Q}} \sign(Q(x))= \sigma(Q)\} . 
		\]
		The sign condition $\sigma$ is \emph{realizable}  if $\RR ( \sigma )$ is non-empty.
		We denote by $\Sign(\mathcal{Q})$ the set of realizable sign conditions of $\mathcal{Q}$.
	\end{definition}

	Let $R \in \R$ with $R> 0$,
	and let 
	\[
	\mathcal{P} = \{P_0, P_1,\ldots,P_s\} \subset \R[X_1,\ldots,X_k],
	\]
	with $P_0 = X_1^2 +\cdots +X_k^2 - R$. Let $P \in \R[X_1,\ldots,X_k]$, and also let
	$\phi$ be a closed
	$(\mathcal{P} - \{P_0\})$-formula, and $\widetilde{\phi}$ be 
	$\phi  \wedge (P_0 \leq 0) \wedge (P - Y \leq 0)$, where $Y$ is a new variable.
	So $\phi$ is a $(\mathcal{P} \cup \{P -Y\})$-closed formula.
	Let $P_{s+1} = P - Y$.
	
	\begin{notation}
		\label{not:start}
		For $\bar{\eps} = (\eps_0,\ldots,\eps_{s+1})$, we denote by
		$\phi^\star(\bar{\eps})$, the $\mathcal{P}^\star(\bar{\eps})$-closed formula
		obtained by replacing each occurrence of $P_i \geq 0$ in $\phi$ by $P_i + \eps_i \geq 0$ 
		(resp. $P_i \leq 0$ in $\widetilde{\phi}$ by $P_i - \eps_i \leq 0$) for $0 \leq i \leq s+1$,
		where 
		\[
		\mathcal{P}^\star(\bar{\eps}) = \bigcup_{0 \leq i \leq s+1} \{ P_i +\eps_i, P_i - \eps_i \}.
		\]
	\end{notation}
	
	Observe that
	\[
	S^\star(\bar{\eps}):= \RR(\phi^\star(\bar{\eps})) \subset \R\la\bar{\eps}\ra^{k+1}
	\] 
	is a $\mathcal{P}^\star(\bar{\eps})$-closed semi-algebraic set,
	and we define $\Sigma_\phi \subset 
	\{-1,0,1\}^{\mathcal{P}^\star(\bar{\eps})}$ by
	\begin{equation}
		\label{eqn:Sigma}
		S^\star(\bar{\eps}) =\bigcup_{ \sigma \in \Sigma_{\phi}, \RR(\sigma) \neq \emptyset} \RR(\sigma).
	\end{equation}

	\begin{lemma}
		\label{lem:non-singular}
		For each $\mathcal{Q} \subset \mathcal{P}^\star(\bar\eps)$,
		$\ZZ(\mathcal{Q},\R\la\bar\eps\ra^{k+1})$ is either empty or 
		is a non-singular $(k+1-\card(\mathcal{Q}))$-dimensional 
		real variety
		such that at every point $(x_1,\ldots,x_k,y) \in \ZZ(\mathcal{Q},\R\la\bar\eps\ra^{k+1}) $, the
		$(\card(\mathcal{Q}) \times (k+1))$-Jacobi matrix,
		\[
		\left( \frac{\partial P}{\partial X_i} , \frac{\partial P}{\partial Y}
		\right)_{P \in \mathcal{Q}\
			1\leq i \leq k}
		\]
		has the maximal rank $\card(\mathcal{Q})$.
	\end{lemma}
	
	\begin{proof}
		See \cite{BV06}.
	\end{proof}

	Now let $\pi_{k+1}:\R\la\bar\eps\ra^{k+1} \rightarrow \R\la\bar\eps\ra$ denote the projection to the last (i.e. the  $Y$) coordinate, and 
	$\pi_{[1,k]}:\R\la\bar\eps\ra^{k+1} \rightarrow \R\la\bar\eps\ra^k$ denote the projection to the first 
	(i.e. $(X_1,\ldots,X_k)$) $k$ coordinates.
	
	For any semi-algebraic 
	subset $S \subset \R\la\bar\eps\ra^{k+1}$, and $T \subset \R\la\bar\eps\ra$, we denote by
	$S_T= \pi_{[1,k]}(S \cap \pi_{k+1}^{-1}(T))$.
	For $t \in \R\la\eps\ra$, we will denote by
	$S_{ \leq t} = S_{(-\infty,t]}$, and $S_t = S_{\{t\}}$. 
	
	\begin{notation}[Critical points and critical values]
		\label{not:crit}
		For $\mathcal{Q} \subset \mathcal{P}^\star(\bar\eps)$, we denote by $\Crit(\mathcal{Q})$
		the subset of $\ZZ(\mathcal{Q},\R\la\bar\eps\ra^{k+1})$ at which the the Jacobian matrix,
		\[
		\left( \frac{\partial P}{\partial X_i} \right)_{P \in \mathcal{Q},
			1 \leq i \leq k}
		\]
		is not of the maximal possible rank.
		We denote $\crit(\mathcal{Q}) = \pi(\Crit(\mathcal{Q}))$.
	\end{notation}
	
	\begin{lemma}
		\label{lem:sard}
		The set 
		\[
		\bigcup_{\mathcal{Q} \subset \mathcal{P}^\star(\bar\eps)} \crit(\mathcal{Q})
		\]
		is finite.
	\end{lemma}
	
	\begin{proof}
		Follows from Lemma~\ref{lem:non-singular} and the semi-algebraic Sard's lemma (see for example \cite[Theorem 5.56]{BPRbook2}).
	\end{proof}

	\begin{lemma}
		\label{lem:Whitney}
		The partitions
		\begin{eqnarray*}
			\label{partition}
			\R^{k+1} &=& \bigcup_{ \sigma \in {\rm Sign}(\mathcal{P}^\star(\bar\eps))}  \RR(\sigma),\\
			S^\star(\bar\eps) &=& \bigcup_{ \sigma \in \Sigma_\phi} \RR(\sigma),
		\end{eqnarray*}
		are  compatible Whitney stratifications of $\R^{k+1}$ and $S^\star(\bar\eps)$ respectively.
	\end{lemma}
	
	\begin{proof}
		Follows directly from the definition of Whitney stratification (see \cite{GM,CS}),
		and Lemma~\ref{lem:non-singular}.
	\end{proof}
	
	We are now in a position to prove the key mathematical result that allows us to reduce the filtration
	of a semi-algebraic set by the sub-level sets of a polynomial to the case of a finite filtration.
	
	\begin{proposition}
		\label{prop:G}
		Suppose  
		\[\bigcup_{\mathcal{Q} \subset \mathcal{P}^\star(\bar\eps)} \crit(\mathcal{Q}) = \{t_0,\ldots,t_N\},
		\] 
		with $t_0 < t_1 < \cdots < t_N$ (cf. Lemma~\ref{lem:sard}). Then for
		$0 \leq i < N$, 
		$a, b \in \R$ such that $(a,b) \subset (t_i,t_{i+1}) \cap \R$, and for any $c \in (a,b)$, the inclusion 
		$$\RR(\phi(\cdot,a)) \hookrightarrow \RR(\phi(\cdot,c))$$
		is a semi-algebraic homotopy equivalence, 
		
	\end{proposition}	
	
	\begin{proof}
		The proof is an adaptation of a proof of a similar result in \cite{BV06} (Lemma 3.8), 
		though our situation is much simpler.
		It follows from Lemma~\ref{lem:Whitney} that the semi-algebraic set 
		\[
		\widehat{S^\star(\bar\eps)} := S^\star(\bar\eps) \setminus \pi_{k+1}^{-1}(\{t_0,\ldots,t_N\})
		\] 
		is a Whitney-stratified set.
		Moreover, $\pi_{k+1}|_{\widehat{S^\star(\bar\eps)}}$ is a proper stratified submersion.
		By Thom's first isotopy lemma (in the semi-algebraic version, over real closed fields
		\cite{CS}) the map $\pi_{k+1}|_{\widehat{S^\star(\bar\eps)}}$ is a locally trivial fibration.
		
		Now let $0 \leq i < N$.
		It follows that for $a',b' \in \R\la\bar\eps\ra$ with $t_i < a' \leq b' < t_{i+1}$, 
		that there exists a semi-algebraic homeomorphism 
		\[
		\theta_{a',b'}: S^\star(\bar\eps)_{[a',b']}  \rightarrow S^\star(\bar\eps)_{a'} \times [a',b']
		\]
		such that the following diagram commutes.
		\[
		\xymatrix{
			S^\star(\bar\eps)_{[a',b']} \ar[rr]^{\theta_{a',b'}}\ar[rd]^{\pi_{k+1}} && S^\star(\bar\eps)_{a'}  \times [a',b'] \ar[ld]_{\pi_{k+1}}\\
			&\R\la\bar\eps\ra&
		}
		\]
		
		Let 
		\[
		r_{a',b'}: S^\star(\bar\eps)_{b'} \times [a',b'] \rightarrow  S^\star(\bar\eps)_{a'},
		\]
		be the map defined by 
		\begin{eqnarray*}
			r_{a',b'}(x,t) &=&  \pi_{[1,k]}\circ \theta_{a',b'}(x,t) \mbox{ if } t \leq P(x), \\
			&=& x, \mbox{ else}.
		\end{eqnarray*}
		
		Notice, $r_{a',b'}$ is a semi-algebraic continuous map, and moreover for $x \in S^\star(\bar\eps)_{a'}$, $r_{a',b'}(x,a') = x$. Thus, $r_{a',b'}$ is a semi-algebraic deformation  retraction of $ S^\star(\bar\eps)_{b'}$ to $S^\star(\bar\eps)_{a'}$. 
		
		This implies that the inclusion 
		\begin{equation}
			\label{eqn:prop:G:1}
			S^\star(\bar\eps)_{a'} \hookrightarrow S^\star(\bar\eps)_{b'}
		\end{equation}
		is a semi-algebraic homotopy equivalence.

		Now suppose that $a,b \in \R$ with $t_i < a \leq b < t_{i+1}$. 
		$S^\star(\bar\eps)_{a}$ and $S^\star(\bar\eps)_{b}$ are closed and bounded over $\R$, and
		that $S^\star(\bar\eps)_{a} \searrow \RR(\phi(\cdot,a))$,
		$S^\star(\bar\eps)_{b} \searrow \RR(\phi(\cdot,b))$.
		
		Then, it follows from Lemma~\ref{lem:monotone} that 
		the inclusions,
		\begin{equation}
			\label{eqn:prop:G:2}
			\E(\RR(\phi(\cdot,a)),\R\la\bar\eps\ra) \hookrightarrow S^\star(\bar\eps)_{a},
		\end{equation}
		and
		\begin{equation}
			\label{eqn:prop:G:3}
			\E(\RR(\phi(\cdot,b)),\R\la\bar\eps\ra) \hookrightarrow S^\star(\bar\eps)_{b},
		\end{equation}
		are semi-algebraic homotopy equivalences.
		
		Thus, we have the following commutative diagram of inclusions
		
		\[
		\begin{tikzcd}        
			S^\star(\bar\eps)_{a}    \arrow[hook]{r} & S^\star(\bar\eps)_{b}  \\ 
			\E(\RR(\phi(\cdot,a)),\R\la\bar\eps\ra) \arrow[hook]{r} \arrow[hook]{u}& \E(\RR(\phi(\cdot,b)),\R\la\bar\eps\ra)
			\arrow[hook]{u}
		\end{tikzcd}
		\]
		
		in which all arrows other than the bottom inclusion are semi-algebraic homotopy equivalences, and hence so
		is the bottom arrow.
		This implies that the inclusion
		$\RR(\phi(\cdot,a)) \hookrightarrow \RR(\phi(\cdot,b))$ is a semi-algebraic homotopy equivalence by an application of the Tarski-Seidenberg transfer principle (see for example \cite[Chapter 2]{BPRbook2}).
		
		Now assume that $a = t_i$. Using Lemma~\ref{lem:monotone} we have that for all small enough $\eps>0$,
		the inclusion $\RR(\phi(\cdot,a)) \hookrightarrow \RR(\phi(\cdot,a+\eps))$ is a semi-algebraic homotopy equivalence. Moreover, from what has been already shown, the inclusion
		$\RR(\phi(\cdot,a+\eps)) \hookrightarrow \RR(\phi(\cdot,c))$ is a semi-algebraic homotopy equivalence.
		It now follows that $\RR(\phi(\cdot,a)) \hookrightarrow \RR(\phi(\cdot,c))$ is a semi-algebraic homotopy equivalence. This completes the proof.
	\end{proof}

	\begin{lemma}
		\label{lem:removal}
		Let $\mathcal{G} \subset \R[\bar\eps][T]$ be a finite set  of non-zero polynomials and 
		\[
		\{t_0,\ldots,t_N\} \subset \bigcup_{G \in \mathcal{G}}  \ZZ(G,\R\la\bar\eps\ra)
		\]
		with $t_0 < \cdots < t_N$. 
		For $G \in \mathcal{G}$,
		let $G  = \sum_\alpha m_{G,\alpha} G_\alpha$, with $G_\alpha \in\R[T], m_{G,\alpha} \in \R[\bar\eps]$, and let $M(G) = \{\alpha \mid m_{G,\alpha} \neq 0 \}$.
		Let $\mathcal{H} = \bigcup_{G \in \mathcal{G}, \alpha \in M(P)}  \{G_\alpha\}$, and 
		let 
		\[
		\{s_0,\ldots,s_M\} = \bigcup_{H \in \mathcal{H}} \ZZ(H,\R)
		\] 
		with $s_0 < s_1 < \cdots < s_M$.
		Then, for each $i, 0 \leq i < M$, 
		there exists $j, 0 \leq j < N$, such that 
		$(s_i,s_{i+1}) \subset \R$ is contained in $(t_j,t_{j+1}) \cap \R$.
	\end{lemma}
	
	\begin{proof}
		Notice that it follows from the definition of the set $\{s_0,\ldots,s_M\}$ that 
		for any $i, 0 \leq i < M$, the sign condition (cf. Definition~\ref{def:sign-condition})
		realized by $\mathcal{H}$ at $t$
		stays fixed for  all $t \in \R$, 
		such that  $t \in (s_i,s_{i+1})$.
		
		Since for any $t \in \R$, the sign condition realized by $\mathcal{H}$ at $t$
		determines the
		sign condition of $\mathcal{G}$ realized at $t$, it follows that the
		the sign condition (cf. Definition~\ref{def:sign-condition})
		realized by $\mathcal{G}$ at $t$
		also stays fixed for  all $t \in \R$, 
		such that  $t \in (s_i,s_{i+1})$.
		
		Suppose that $t' \in \E((s_i,s_{i+1}),\R\la\bar\eps\ra)$ such that $G(t') = 0$ for some $G \in \mathcal{G}$. We claim that this implies that $\lim_{\bar\eps} t' \in \{s_i,s_{i+1}\}$. Suppose not.
		Then, $\lim_{\bar\eps} t' \in (s_i,s_{i+1})$, which contradicts the fact that 
		the sign condition (cf. Definition~\ref{def:sign-condition})
		realized by $\mathcal{G}$ at $t$
		stays fixed for  all $t \in \R$, 
		such that  $t \in (s_i,s_{i+1})$, since $G$ is a non-zero polynomial.

		The lemma now follows from the hypothesis that 
		$\{t_0,\ldots,t_N\} \subset \bigcup_{G \in \mathcal{G}} \ZZ(G,\R\la\bar\eps\ra)$.
	\end{proof}

	Let $S=\RR(\Phi)$ and $P, t_0,\ldots,t_N$ as in Proposition~\ref{prop:G}, and let $\mathcal{G}, \mathcal{H}$, and 
	$s_0 < \cdots < s_M$ as in Lemma~\ref{lem:removal}.
	Let $s_{-1} = -\infty, s_{M+1} = \infty$. Let 
	$\mathcal{F}$ denote the finite filtration of semi-algebraic sets,
	indexed by the finite ordered set $T = \{s_i \mid -1 \leq i \leq M+1 \}$,
	with the element of $\mathcal{F}$ indexed by $s_i$ equal to $S_{P \leq s_i}$.
	We have the following proposition.

	\begin{proposition}
		\label{prop:G2}
		For each $p \geq 0$,
		\[
		\mathcal{B}_p(S,P) = \mathcal{B}_p(\mathcal{F}).
		\]
	\end{proposition}
	
	\begin{proof}
		It follows from  Proposition~\ref{prop:G} and Lemma~\ref{lem:removal} that for each $i, -1 \leq i \leq  M$ and 
		$s \in (s_i,s_{i+1})$, the inclusion $S_{P \leq s_i} \hookrightarrow S_{P \leq s}$ is a semi-algebraic homotopy
		equivalence.
		
		The proposition 
		will now follow from the following two claims.
		\begin{claim}
			\label{claim:proof:cor:G:1}
			Suppose that $s,t \in [s_{-1}, s_{M+1}], s \leq t$. 
			Then, $\mu_p^{s,t}(\mathcal{F}(S,P)) \neq 0 \Rightarrow s,t \in \{s_{-1},\ldots,s_{M+1}\}$.
		\end{claim}
		
		\begin{proof}
			We consider the following two cases.
			\begin{enumerate}[1.]
				\item $s \not\in \{s_{-1},\ldots,s_{M+1}\}$: Without loss of generality we can assume
				that $s \in (s_i,s_{i+1})$ for some $i, -1 \leq i \leq M$.  Now the inclusion 
				$S_{P \leq s'} \hookrightarrow S_{P \leq s}$,
				is a semi-algebraic homotopy equivalence for all $s' \in [s_{i},s)$,\and hence $i^{s',s}_p$ is an isomorphism
				for all  $s' \in [s_{i},s)$.
				
				It follows that for all $s' \in [s_i,s)$,  
				\[
				\HH^{s',t}_p(\mathcal{F}(S,P)) = \mathrm{Im}(i_p^{s',t}) = \mathrm{Im}(i_p^{s,t}\circ i_p^{s',s}) = \mathrm{Im}(i^{s,t}_p) = \HH^{s,t}_p(\mathcal{F}(S,P)),
				\]
				which implies that
				\[
				(i^{s,t}_p)^{-1}(\HH^{s',t}_p(\mathcal{F}(S,P))) = (i^{s,t}_p)^{-1}(\HH^{s,t}_p(\mathcal{F}(S,P))) =  \HH_p(S_{P \leq s}).
				\]
				
				Noting that
				\[
				\bigcup_{s' < s} (i^{s,t}_p)^{-1}(\HH^{s',t}_p(\mathcal{F}(S,P))) = \bigcup_{s' \in [s_i,s)} (i^{s,t}_p)^{-1}(\HH^{s',t}_p(\mathcal{F}(S,P))),
				\]
				it now follows that
				\begin{eqnarray*}
					M^{s,t}_p(\mathcal{F}(S,P)) &=& \bigcup_{s' < s} (i^{s,t}_p)^{-1}(\HH^{s',t}_p(\mathcal{F}(S,P))) \\
					&=& \HH_p(S_{P \leq s}), \\
					N^{s,t}_p(\mathcal{F}(S,P)) &=& \bigcup_{s' < s\leq t' < t} (i^{s,t'}_p)^{-1}(\HH^{s',t'}_p(\mathcal{F}(S,P))) \\
					&=&
					\bigcup_{s\leq t' < t} (i^{s,t'}_p)^{-1}(\HH^{s,t'}_p(\mathcal{F}(S,P))) \\
					&=& \HH_p(S_{P \leq s}).
				\end{eqnarray*}
				
				We have two sub-cases to consider.
				\begin{enumerate}[(a)]
				    \item 
				    If $t < s_{M+1}$: 
				    \[
				P^{s,t}_p(\mathcal{F}(S,P)) = M^{s,t}_p(\mathcal{F}(S,P))/N^{s,t}_p(\mathcal{F}(S,P)) = 0.
				    \]
				    \item
				    If $t = s_{M+1} = \infty$:
				    \[
				P^{s,\infty}_p(\mathcal{F}(S,P)) =  \HH_p(S_{P \leq s}) / \bigcup_{s \leq t} M^{s,t}_p(\mathcal{F}(S,P)) = 0.
				\]
				since 
				\[
				\bigcup_{s \leq t} M^{s,t}_p(\mathcal{F}(S,P)) = \bigcup_{s \leq t} \HH_p(S_{P \leq s}) = \HH_p(S_{P \leq s}).
				\]
				\end{enumerate}

				\item
				$t \not\in \{s_{-1},\ldots,s_{M+1}\}$: Without loss of generality we can assume
				that $t \in (s_i,s_{i+1})$ for some $i, -1 \leq i \leq M$.
				The inclusion $S_{P \leq t'} \hookrightarrow S_{P \leq t}$,
				is a semi-algebraic homotopy equivalence for all $t' \in [s_i,t)$, and
				hence $i_p^{t',t}$ is an isomorphism for all $t' \in [s_i,t)$.
				This implies that for all $t' \in [s_i,t)$, and $s' < t'$,
				$\mathrm{Im}(i^{s',t'}_p)$ can be identified with $\mathrm{Im}(i^{s',t}_p)$ using the isomorphism
				$i^{t',t}_p$.
				Furthermore, it is easy to verify that for every fixed $s'<s$ and $s \leq t' \leq t''$,
				\[
				(i^{s,t'}_p)^{-1}(\HH^{s',t'}_p(\mathcal{F}(S,P))) \subset (i^{s,t'}_p)^{-1}(\HH^{s',t''}_p(\mathcal{F}(S,P))),
				\]
				and hence for each fixed $s' < s$,
				\[
				\bigcup_{s\leq t' < t} (i^{s,t'}_p)^{-1}(\HH^{s',t'}_p(\mathcal{F}(S,P))) 
				=
				\bigcup_{s_i < t' < t} (i^{s,t'}_p)^{-1}(\HH^{s',t'}_p(\mathcal{F}(S,P))).
				\]
				
				It follows that for $t \in (s_i,s_{i+1})$
				\begin{eqnarray*}
					N^{s,t}_p(\mathcal{F}(S,P)) &=& \bigcup_{s' < s\leq t' < t} (i^{s,t'}_p)^{-1}(\HH^{s',t'}_p(\mathcal{F}(S,P))) \\
					&=& \bigcup_{s' < s} (i^{s,t}_p)^{-1}(\HH^{s',t}_p(\mathcal{F}(S,P))) \\
					&=&  M^{s,t}_p(\mathcal{F}(S,P)).
				\end{eqnarray*}
				
			We have
				\[
				P^{s,t}_p(\mathcal{F}(S,P)) = M^{s,t}_p(\mathcal{F}(S,P))/N^{s,t}_p(\mathcal{F}(S,P))  = 0.
				\]
			\end{enumerate}
			This completes the proof.
		\end{proof}
		
		\begin{claim}
			\label{claim:proof:cor:G:2}
			For each $i,j, -1 \leq i \leq j \leq M+1$, 
			$\mu_p^{s_i,s_j}(\mathcal{F}(S,P)) = \mu_p^{s_i,s_j}(\mathcal{F})$.
		\end{claim}
		
		\begin{proof}
			It suffices to prove that 
			\begin{eqnarray*}
				M^{s_i,s_j}_p(\mathcal{F}(S,P)) &=& M^{s_i,s_j}_p(\mathcal{F}), \\
				N^{s_i,s_j}_p(\mathcal{F}(S,P)) &=& N^{s_i,s_j}_p(\mathcal{F}).
			\end{eqnarray*}
			
			To prove the first equality we use the fact that
			$s' \in [s_{i-1},s_{i})$, the inclusion $S_{P \leq s_{i-1}} \hookrightarrow S_{P \leq s'}$ is a semi-algebraic homotopy equivalence.
			
			Hence,
			\begin{eqnarray*}
				M^{s_i,s_j}_p(\mathcal{F}(S,P)) &=& \bigcup_{s' < s_i} (i^{s_i,s_j}_p)^{-1}(\HH^{s',s_j}_p(\mathcal{F}(S,P))) \\
				&=& (i^{s_{i},s_j}_p)^{-1}(\HH^{s_{i-1},s_j}_p(\mathcal{F}(S,P))) \\
				&=& M^{s_i,s_j}_p(\mathcal{F}).
			\end{eqnarray*}
			
			Using additionally the fact that 
			$t' \in [s_{j-1},s_{j})$, the inclusion $S_{P \leq s_{j-1}} \hookrightarrow S_{P \leq t'}$ is a semi-algebraic homotopy equivalence, we have: 
			\begin{eqnarray*}
				N^{s_i,s_j}_p(\mathcal{F}(S,P)) &=& \bigcup_{s' < s_i\leq t' < s_j} (i^{s_i,t'}_p)^{-1}(\HH^{s',t'}_p(\mathcal{F}(S,P))) \\
				&=& (i^{s_i,s_{j-1}}_p)^{-1}(\HH^{s_{i-1},s_{j-1}}_p(\mathcal{F}(S,P)))\\
				&=& N^{s_i,s_j}_p(\mathcal{F}).
			\end{eqnarray*}
		\end{proof}
		This concludes the proof of Proposition~\ref{prop:G2}.
	\end{proof}
	
	\begin{proof}[Proof of Proposition~\ref{prop:finite}]
		Follows immediately from Proposition~\ref{prop:G2}.
	\end{proof}
	
	\subsection{Persistent multiplicities for finite filtration}
	\label{subsec:multiplicities}
		In this section, we prove a formula for the persistent multiplicities associated to a finite filtration $\mathcal{F}$, which we later use in Algorithm~\ref{alg:barcode-simplicial} to obtain the barcodes of a finite filtration. 
	We deduce the formula from our definition of persistent multiplicity (cf. Eqn.~\eqref{eqn:def:barcode:multiplicity} in  Definition~\ref{def:barcode_multiplicity}).~\footnote{This formula already appears in \cite[page 152]{Edelsbrunner-Harer}, but what is
		meant by ``independent $p$-dimensional classes that are born at $K_i$, and die entering $K_j$''
		\emph{loc. cit.}  is not totally transparent. See also Remark~\ref{rem:subspaces}.}
	
	\begin{proposition}
		\label{prop:finite:multiplicity}
		Let $\mathcal{F}$ denote a finite filtration, given by $ X_0 \subset \cdots \subset X_{M} = X_{M+1}= \cdots = X$, 
		such that rank of $\HH_p(X_j)$ is finite for each $p \geq 0$. Then for $0 <j < k$, 
		
		\begin{equation}
			\label{eqn:prop:finite:multiplicity}
			\mu^{j,k}_p(\mathcal{F}) = 	
			\left\{
			\begin{array}{ll}
				(b^{j, {k-1}}_p (\mathcal{F}) - b^{j, k}_p
				(\mathcal{F})) - (b^{{j-1},{k-1}}_p (\mathcal{F}) - b^{{j-1}, k}_p(\mathcal{F})), & k < \infty, \\ 
				& \\
				b^{j, k}_p
				(\mathcal{F}) - b^{{j-1},{k}}_p (\mathcal{F}), & k=\infty. \\
			\end{array} 
			\right.		
		\end{equation}
		\\
	\end{proposition}
	
	\begin{proof}
		We first prove the case where $k$ is finite. By Definition~\ref{def:barcode_subquotients},
		\begin{eqnarray*}
			\mu^{j,k}_p(\mathcal{F}) &=& \dim P^{j,k}_p(\mathcal{F}) \\
			&=& \dim M^{j,k}_p(\mathcal{F}) - \dim N^{j,k}_p(\mathcal{F}).
		\end{eqnarray*}
		
		Since $\mathcal{F}$ is finite, we have
		\begin{eqnarray*}
			M^{j,k}_p(\mathcal{F}) &=& (i^{j,k}_p)^{-1}(\HH^{j-1,k}_p(\mathcal{F})), \\
			N^{j,k}_p(\mathcal{F}) &=& (i^{j,k-1}_p)^{-1}(\HH^{j-1,k-1}_p(\mathcal{F})).
		\end{eqnarray*}	
		Note that  $(i^{j,k}_p)^{-1}(\HH^{j-1,k}_p(\mathcal{F}))$ is a subspace of  $\HH_p(X_j)$,
		and  hence the linear map $i^{j,k}: \HH_p(X_j) \rightarrow \HH_p(X_k)$ factors through a
		surjection $f: \HH_p(X_j) \rightarrow \HH^{j,k}_p(\mathcal{F})$ followed by an injection
		$\HH^{j,k}_p(\mathcal{F}) \hookrightarrow \HH_p(X_k)$ as shown in the following diagram.

		\begin{center}
			\begin{tikzcd}[column sep=scriptsize] \HH_p(X_j) \arrow[dr, "f"] \arrow[rr, "i^{j,k}_p"]{}& & \HH_p(X_k) \\& \HH^{j,k}_p(\mathcal{F}) \arrow[hookrightarrow]{ur} 
			\end{tikzcd}.
		\end{center}
		
		Now $\HH^{j-1,k}_p(\mathcal{F})$ is a subspace of $\HH^{j,k}_p(\mathcal{F})$, 
		and let \[
		m:\HH^{j,k}_p(\mathcal{F}) \rightarrow \HH^{j,k}_p(\mathcal{F})/\HH^{j-1,k}_p(\mathcal{F})
		\]
		be the canonical surjection.
		Let $g = m \circ f$. Since $f$ and $m$ are both surjective, so is $g$.
		\begin{center}
			\begin{tikzcd}
				\HH_p(X_j)\ar[r,"f"]\ar[rr,out=-30,in=210,swap,"g \ = \ m \circ f"] & \HH^{j,k}_p(\mathcal{F})\ar[r,"m"] & \HH^{j,k}_p(\mathcal{F}) \ / \ \HH^{j-1,k}_p(\mathcal{F})
			\end{tikzcd},
		\end{center}
		Now notice that
		\begin{eqnarray*}
			M^{j,k}_p(\mathcal{F}) &=& (i^{j,k}_p)^{-1}(\HH^{j-1,k}_p(\mathcal{F})) \\
			&=& f^{-1}(\HH^{j-1,k}_p(\mathcal{F})) \\
			&=& \ker(g).
		\end{eqnarray*}	
		Since $g$ is surjective,
		\[
		\rank(g) = \dim \HH^{j-1,k}_p(\mathcal{F}) -\dim 
		\HH^{j,k}_p(\mathcal{F}),
		\]
		and using the rank-nullity theorem
		we obtain 
		\begin{equation}\label{eq:dim_M}
			\dim M^{j,k}_p(\mathcal{F}) = b_p(X_j) - (b^{j, k}_p(\mathcal{F})- b^{j-1, k}_p(\mathcal{F})).
		\end{equation}
		
		Using a similar argument we obtain
		\begin{equation} \label{eq:dim_N}	
			\dim N^{j,k}_p(\mathcal{F}) = b_p(X_j) - (b^{j, k-1}_p(\mathcal{F})- b^{j-1, k-1}_p(\mathcal{F})).	
		\end{equation}

		Finally,
		\begin{eqnarray*}
			\mu^{j,k}_p(\mathcal{F})
			&=& \dim M^{j,k}_p(\mathcal{F}) - \dim N^{j,k}_p(\mathcal{F})\\
			&=& b^{j-1, k}_p(\mathcal{F}) -b^{j, k}_p(\mathcal{F}) + (b^{j, k-1}_p(\mathcal{F}) -b^{j-1, k-1}_p(\mathcal{F}))  \\	 
			&=& (b^{j, k-1}_p(\mathcal{F}) -b^{j, k}_p(\mathcal{F})) - (b^{j-1, k-1}_p(\mathcal{F}) - b^{j-1, k}_p(\mathcal{F})).
		\end{eqnarray*}
		If $k=\infty$, then by Definition~\ref{def:barcode_subquotients}, 
		\begin{eqnarray*}
			\mu^{j,k}_p(\mathcal{F}) &=& \dim P^{j,k}_p(\mathcal{F}) \\
			&=& \dim \HH_p(K_j) - \dim \bigcup_{j \leq t} M^{j,t}_p(\mathcal{F}).
		\end{eqnarray*}
		Since $M^{s,t}_p(\mathcal{F}) \subset M^{s,t'}_p(\mathcal{F})$ for $t \leq t'$,  we have 
		$$M^{j,t}_p(\mathcal{F}) \subset M^{j,t+1}_p(\mathcal{F}) \subset \cdots \subset M^{j,M}_p(\mathcal{F}) = M^{j,M+1}_p(\mathcal{F})= \cdots = M^{j,\infty}_p(\mathcal{F})$$ \vspace*{-.4cm}
		$$\bigcup_t^\infty M^{j,t}_p(\mathcal{F}) = M^{j,M}_p(\mathcal{F})$$
		Therefore,
		\begin{eqnarray*}
			\mu^{j,k}_p(\mathcal{F}) &=& \dim \HH_p(K_j) - \dim M^{j,M}_p(\mathcal{F}) \\
			&=& b_p(X_j) - (b_p(X_j) - (b^{j, M}_p(\mathcal{F})- b^{j-1, M}_p(\mathcal{F}))) \\
			&=& b^{j, M}_p(\mathcal{F})- b^{j-1, M}_p(\mathcal{F})
		\end{eqnarray*}
	\end{proof}
	
	\section{Algorithms and proof of Theorem~\ref{thm:persistent}}
	\label{sec:finite-algorithm}
	In this section we describe our algorithmic results leading to the
	proof of Theorem~\ref{thm:persistent}.
	We begin by stating some preliminary mathematical results in Section~\ref{subsec:preliminaries} that we will need for our algorithms. We describe two technical algorithms that we will need 
	in Section~\ref{subsec:algo-prelim}. 
	In Section~\ref{subsec:algo-finite} we describe Algorithm~\ref{alg:filtration} for reducing the given continuous filtration to a finite one. The proof of correctness of this algorithm relies on Proposition~\ref{prop:G2} proved earlier.
	Finally, in Section~\ref{subsec:algo-finite} we describe our algorithm for computing the barcode of a semi-algebraic filtration
	(algorithm~\ref{alg:barcode-semi-algebraic}), prove its correctness and analyze its complexity, thereby proving Theorem~\ref{thm:persistent}.

	\subsection{Preliminaries}
	\label{subsec:preliminaries}

	\begin{notation} [Derivatives] 
		\label{not:ders}
		Let $P$ be a univariate polynomial of degree  $p$ in $\R[X]$. We  will denote by $\Der (P)$ the tuple  
		$(P,P',\ldots,P^{(p)})$ of derivatives of $P$.
	\end{notation}
	
	The significance of $\Der(P)$ is encapsulated in the following lemma which underlies
	our representations of elements of $\R$ which  are algebraic over $\D$ 
	(cf. Definition~\ref{10:def:thom}).

	\begin{proposition}[Thom's Lemma]
		\label{prop:Thom}
		Let $f \in \R[X]$ be a univariate polynomial, and, let $\sigma$ be
		a sign condition on $\Der (f)$ Then $\RR ( \sigma )$ is either empty, a
		point, or an open interval.
	\end{proposition}
	
	\begin{proof}
		See \cite[Proposition 2.27]{BPRbook2}.
	\end{proof}
	
	Proposition~\ref{prop:Thom} allows us to specify elements of $\R$ which are algebraic over $\D$ by means
	of a pair $(f,\sigma)$ where $f \in \D[X]$ and $\sigma \in \{0,1,-1\}^{\Der(f)}$.
	
	\begin{definition}
		\label{10:def:thom}
		We say that $x \in \R$ is \emph{associated to the pair $(f,\sigma)$},
		if $\sigma(f)=0$ and
		if $\Der(f)$ realizes the sign condition $\sigma$  at $x$.
		We call the pair $(f,\sigma)$ to be a  \emph{Thom encoding} specifying $x$.
	\end{definition}

	We will also use the notion of a weak sign condition (cf. Definition~\ref{def:sign-condition}).
	
	\begin{definition}
		\label{def:weaksign}
		A \emph{weak sign condition} is an element of
		\[ \{\{0\},\{0,1\},\{0, -1\}\} . \]
		We say
		\[ \begin{cases}
			\sign (x) \in \{0\} & \mbox{if and only if } 
			x=0,\\
			\sign (x) \in \{0,1\} & \mbox{if and only if }  x \ge 0,\\
			\sign (x) \in \{0, - 1\} & \mbox{if and only if }  x \le 0.
		\end{cases} \]
		A \emph{weak sign condition} on $\mathcal{Q}$ is an element of
		$\{\{0\},\{0,1\},\{0, - 1\}\}^{\mathcal{Q}}$. If $\sigma \in \{0,1, - 1\}^{\mathcal{Q}}$, its
		\emph{relaxation}
		$\overline{\sigma}$  is the weak sign condition on
		$\mathcal{Q}$ defined by $\overline{\sigma} (Q) = \overline{\sigma (Q)}$.
		The \emph{realization of the weak sign condition $\tau$}
		is
		\[ \RR ( \tau ) = \{x \in \R^{k}   \mid  
		\bigwedge_{Q \in \mathcal{Q}} \sign (Q(x)) \in \tau (Q)\} . \]
	\end{definition}
	
	\begin{definition}
		We say that a set of polynomials $\mathcal{F} \subset \R[X]$ is \emph{closed under differentiation} 
		if $0 \not \in \mathcal{F}$ and if for each $f \in
		\mathcal{F}$ then $f' \in \mathcal{F}$ or $f' =0$. 
	\end{definition}
	
	\begin{lemma}(\cite[Lemma 5.33]{BPRbook2})
		\label{lem:Thom's Lemma} Let
		$\mathcal{F} \subset \R [X]$ be a finite set of polynomials closed under
		differentiation and let $\sigma$ be a sign condition on the set
		$\mathcal{F}$. Then
		\begin{enumerate}[(a)]
			\item $\RR ( \sigma )$ is either empty, a point, or an open interval.
			\item If $\RR ( \sigma )$ is empty, then $\RR ( \overline{\sigma} )$ is
			either empty or a point.
			\item If $\RR ( \sigma )$ is a point, then $\RR ( \overline{\sigma} )$ is
			the same point.
			\item If $\RR ( \sigma )$ is an open interval then $\RR (
			\overline{\sigma} )$ is the corresponding closed interval.
		\end{enumerate}
	\end{lemma}

	\begin{remark}
		\label{rem:abuse}
		In what follows we will allow ourselves to use for $P \in \R[X_1,\ldots,X_k]$, 
		$\sign(P) = 0$ (resp. $\sign(P) =1$, $\sign(P) = -1$) in place of the atoms $P=0$ (resp. $P > 0$, $P < 0$)
		in formulas. 
		Similarly, we might write $\sign(P) \in \bar{\sigma}$, where $\bar{\sigma}$ is a weak sign condition
		in place of the corresponding weak inequality $P \geq 0$ or $P \leq 0$. It should be clear that this abuse of notation is harmless.
	\end{remark}

	In addition to the mathematical preliminaries described above, we also need two technical algorithmic results that we describe in the next section
	
	\subsection{Some preliminary algorithms}
	\label{subsec:algo-prelim}
	
	For technical reasons that will become clear when we describe Algorithm~\ref{alg:filtration}, we will need
	to convert efficiently a given quantifier-free formula defining a closed semi-algebraic set, into a closed formula
	defining the same semi-algebraic set. This is a non-trivial problem, since the standard quantifier-elimination
	algorithms in algorithmic semi-algebraic geometry does not guarantee that the output will be a closed
	formula even if it is known in advance that the semi-algebraic set that the formula is describing is closed.
	Luckily we only need to deal with formulas in one variable, where the problem is somewhat simpler. Note 
	that even in this case, it is not possible to obtain the description of the given closed semi-algebraic set as
	a closed formula by merely weakening the inequalities in the original formula.

	For example, consider the formula 
	$\phi(X) := (X^2(X-1) > 0) \wedge ((X \geq 2) \vee (X \leq 0))$. Then, $\RR(\phi) = [2,\infty)$
	is a closed semi-algebraic set, but the formula obtained by weakening the inequality $X^2(X-1) > 0$, namely
	\[
	\widetilde{\phi} := (X^2(X-1) \geq 0) \wedge ((X \geq 2) \vee (X \leq 0)),
	\]
	has as its realization the set $\{0\} \cup [2,\infty)$ which is strictly bigger than $\RR(\phi)$.
	
	Nevertheless, using Lemma~\ref{lem:Thom's Lemma}  we have the following algorithm to achieve the above
	mentioned task efficiently.
	
	\begin{algorithm}[H]
		\caption{(Make closed)}
		\label{alg:make-closed}
		\begin{algorithmic}[1]
			\INPUT
			\Statex{
				A quantifier-free formula $\theta(Y)$ with coefficients in $\D$, in one free variable $Y$, such that $\RR(\theta)$ is closed.
			}
			\OUTPUT
			\Statex{
				A 
				closed formula $\psi(Y)$ equivalent to $\theta(Y)$.
			}

			\PROCEDURE
			\State{Let $\theta(Y) = \bigvee_{1 \leq i \leq M} \bigwedge_{1 \leq j \leq N_i} (\sign(F_{i,j}) = \sigma_{i,j})$.}

			\For{each $(i,j)$ such that $\sigma_{i,j} \neq 0$}
			\State{Call Algorithm 13.1 (Computing realizable sign conditions) in \cite{BPRbook2} with input $\Der(F_{i,j})$, and obtain the set $\Sigma_{i,j}$ of realizable sign conditions of $\Der(F_{i,j})$.}
			\State{$\Sigma'_{i,j} \leftarrow \{\sigma \in \Sigma_F \mid \sigma(F_{i,j}) = \sigma_{i,j} \}$.}
			\State{$\overline{\Sigma_{i,j}} \leftarrow \{\bar{\sigma} \mid \sigma \in \Sigma'_{i,j}\}$.}
			\EndFor
			
			\State\Return{the formula 
				\[
				\psi(Y) =  \bigvee_{1 \leq i \leq M} (\bigwedge_{\sigma_{i,j} = 0} (\sign(F_{i,j}) = 0) \wedge \bigwedge_{\sigma_{i,j} \neq 0} \bigvee_{\bar{\sigma} \in \overline{\Sigma_{i,j}}}  (\sign(F_{i,j}) \in \bar{\sigma}).
				\] 
			}
			\COMPLEXITY
			The complexity of the algorithm is 
			bounded by 
			$(s d)^{O(1)}$ where $s$ is the number of polynomials appearing $\theta$ and $d$ a bound on their degrees.
		\end{algorithmic}
	\end{algorithm} 
	
	\begin{proof}[Proof of correctness]
		The correctness of the algorithm follows from the correctness of Algorithm 13.1 (Computing realizable sign conditions) in \cite{BPRbook2}, and Lemma~\ref{lem:Thom's Lemma}.
	\end{proof}
	
	\begin{proof}[Complexity analysis]
		The complexity bound follows from the complexity of Algorithm 13.1 (Computing realizable sign conditions) in \cite{BPRbook2}.
	\end{proof}
	
	We will also need an algorithm that takes as input a finite set of polynomials $\mathcal{G}$ in one variable with
	coefficients in $\D[\bar\eps]$, and outputs a set of Thom encodings whose set of associated 
	points $\{s_0,\ldots,s_M\}$
	satisfy the property stated in Lemma~\ref{lem:removal}.

	\begin{algorithm}[H]
		\caption{(Removal of infinitesimals)}
		\label{alg:removal}
		\begin{algorithmic}[1]
			\INPUT
			\Statex{
				A  finite set $\mathcal{G} \subset \D[\bar\eps][T]$ such that each $P \in \mathcal{G}$ depends on at most
				$k+1$ of the $\eps_i$'s. 
			}
			
			\OUTPUT
			\Statex{
				
				A finite set of Thom encodings $\mathcal{F} = \{(f_i,\sigma_i) \mid 0 \leq i \leq N \}$, with $f_i \in \D[T]$
				with associated points $s_0 < \cdots < s_M$, such that 
				letting $s_{-1} = -\infty, s_{M+1} = \infty$,
				for each $i, 0 \leq i < M$, 
				there exists $j, 0 \leq j < N$, such that 
				$(s_i,s_{i+1}) \subset \R$ is contained in $(t_j,t_{j+1}) \cap \R$,
				where $\{t_0,\ldots,t_N\} = \bigcup_{G \in \mathcal{G}} \ZZ(G,\R\la\bar\eps\ra)$, with $t_0 < \cdots < t_N$.
			}

			\algstore{myalg}
		\end{algorithmic}
	\end{algorithm}
	
	\begin{algorithm}[H]
		\begin{algorithmic}[1]
			\algrestore{myalg}

			\PROCEDURE
			\For {$G \in \mathcal{G}$}
			\State{ $0 \leq i_0 < \cdots < i_h \leq s+1$ be such that $G \in D[\eps_{i_0},\ldots, \eps_{i_h}][T]$.}
			\State{ Write $G = \sum_\alpha m_{G,\alpha}(\eps_{i_0},\ldots, \eps_{i_h})  G_\alpha$, with $G_\alpha \in\D[T], m_{G,\alpha} \in D[\eps_{i_0},\ldots, \eps_{i_h}]$.}
			\State{Let $M(G) = \{\alpha \mid m_{G,\alpha} \neq 0 \}$.}
			\EndFor
			
			\State{Let $\mathcal{H} = \bigcup_{G \in \mathcal{G}, \alpha \in M(G)}  \{G_\alpha\}$.}
			
			\State{ Use Algorithm 10.17 from \cite{BPRbook2} with $\mathcal{H}$ as input to obtain an 
				ordered list of Thom encodings $\mathcal{F}$.}

			\State\Return{$\mathcal{F}$.}

			\COMPLEXITY
			The complexity of the algorithm is 
			bounded by $s D^{O(k)}$, where $s = \card(\mathcal{G})$ and $D$ is a bound on the degrees of the
			polynomials in $\mathcal{G}$ in $\bar\eps$ and in $T$.
		\end{algorithmic}
	\end{algorithm}
	
	\begin{proof}[Proof of correctness]
		The correctness of the algorithm follows from Lemma~\ref{lem:removal} and the correctness of 
		Algorithm 10.17 from \cite{BPRbook2}.
	\end{proof}
	
	\begin{proof}[Complexity analysis]
		The complexity bound follows from the complexity bound of Algorithm 10.17 from \cite{BPRbook2}.
	\end{proof}

\subsection{Algorithm for computing simplicial replacement}
\label{subsec:simplicial-replacement}
We recall the following definition from \cite{basu-karisani}.

\begin{notation} [Diagram of various unions of a finite number of subspaces]
\label{not:diagram-Delta}
Let $J$ be a finite set, $A$ a topological space, 
and $\mathcal{A} = (A_j)_{j \in J}$ a tuple of subspaces of $A$  indexed by $J$.

For any subset 
$J' \subset J$,
we denote 
\begin{eqnarray*}
\mathcal{A}^{J'} &=& \bigcup_{j' \in J'} A_{j'}, \\
\mathcal{A}_{J'} &=& \bigcap_{j' \in J'} A_{j'}, \\
\end{eqnarray*}

We consider $2^J$ as a category whose objects are elements of $2^J$, and whose only morphisms 
are given by: 
\begin{eqnarray*}
2^J(J',J'') &=& \emptyset  \mbox{ if  } J' \not\subset J'', \\
2^J(J',J'') &=& \{\iota_{J',J''}\} \mbox{  if } J' \subset J''.
\end{eqnarray*} 
We denote by $\Simp^J(\mathcal{A}):2^J \rightarrow \Top$ the functor (or the diagram) defined by
\[
\Simp^J(\mathcal{A})(J') = \mathcal{A}^{J'}, J' \in 2^J,
\]
and
$\Simp^J(\mathcal{A})(\iota_{J',J''})$ is the inclusion map $\mathcal{A}^{J'} \hookrightarrow \mathcal{A}^{J''}$.
\end{notation}

We will use an algorithm whose existence is proved in \cite[Theorem 1]{basu-karisani},
and which we will refer to as \emph{Algorithm for computing simplicial replacement},
that given a tuple of closed-formulas 
$\Phi = (\phi_0,\ldots,\phi_N)$, $R >0$,  and $\ell \geq 0$,
produces as output 
a simplicial complex $K$
and subcomplexes $K_i, 0\leq i \leq N$ of $K$, such that 
the diagram
\[
\Simp^{[N]}\left((\RR(\phi_i,\overline{B_k(0,R)}))_{i \in [N]}\right) 
\]
is homologically $\ell$-equivalent (\cite[Section 2.1.1]{basu-karisani}) 
to the diagram
\[
\Simp^{[N]}\left((|K_i|)_{i \in [N]}\right)
\]
(where $|K_i| \subset |K|$ is the geometric realization of $K_i$ and
$[N] = \{0,\ldots,N\}$).

We refer the reader to \cite{basu-karisani} for the details.

The complexity of this algorithm, as well as the size of the output
 simplicial complex $\Delta$, 
are bounded by 
\[
 (N s d)^{k^{O(m)}},
\]
where $s = \card(\mathcal{P})$, and $d = \max_{P \in \mathcal{P}} \deg(P)$.

	\subsection{Algorithm for reducing to a finite filtration}
	\label{subsec:algo-finite}
	We are now in a position to describe our algorithm for reducing the problem of 
	computing the barcode of 
	a filtration of a semi-algebraic set $S$ by the sub-level sets of 
	a polynomial $P$, to the problem of computing the barcode of a finite filtration.
	
	Algorithm~\ref{alg:filtration} computes a finite subset of $\R$, as Thom encodings (cf. Definition~\ref{10:def:thom}), such that it includes
	the values of $P$ at which the homotopy type of the sub-level sets of $S$ changes. The algorithm has singly exponentially bounded complexity.

	\begin{algorithm}[H]
		\caption{(Reducing to a finite filtration)}
		\label{alg:filtration}
		\begin{algorithmic}[1]
			\INPUT
			\Statex{
				\begin{enumerate}[(a)]
					\item $\ell \in \Z_{\geq 0}$.
					\item
					$R \in \D, R >0$.
					\item
					A finite set $\mathcal{P} = \{P_1,\ldots,P_s \} \subset \D[X_1,\ldots,X_k]$.
					\item
					A $\mathcal{P}$-closed formula $\phi$.
					\item
					A polynomial $P \in \R[X_1,\ldots,X_k]$.
				\end{enumerate}
			}

			\OUTPUT
			\Statex{
				\begin{enumerate}[(a)]
					\item
					A finite set of Thom encodings $\mathcal{F} = \{(f_i,\sigma_i) \mid 0 \leq i \leq N \}$, with $f_i \in \D[T]$
					with associated points $t_0 < \cdots < t_N$, such that  for $t \in \R$, denoting by
					$S_{t}= \RR(\phi) \cap \overline{B_k(0,R)} \cap \{x \mid P(x) \leq t\}$,  
					for each $i, 0 \leq i \leq N-1$, and all $t \in [t_i, t_{i+1})$ the inclusion maps $S_{t_i}\hookrightarrow S_{t}$ are homological
					equivalences.
					\item
					A filtration of finite simplicial complexes 
					\[
					K_0 \subset K_1 \subset \cdots \subset K_N
					\] 
					such that 
					$\Simp^{[N]}(S_{t_0},\ldots,S_{t_N})$ is homologically $\ell$-equivalent to $\Simp^{[N]}(|K_0|,\ldots,|K_N|)$.

				\end{enumerate}
			}

			\PROCEDURE
			\State{
				$P_0 \leftarrow \sum_{i=1}^k X_i^2 - R$.
			}
			\State
			{$P_{s+1}  \leftarrow P - Y$.
			}

			\State{
				\[
				\mathcal{P}^\star(\bar{\eps}) \leftarrow  \bigcup_{0 \leq i \leq s+1} \{ P_i +\eps_i, P_i - \eps_i \}.
				\]
			}

			\State{
				Denote by
				$\phi^\star(\bar{\eps})$, the $\mathcal{P}^\star(\bar{\eps})$-closed formula
				obtained by replacing each occurrence of $P_i \geq 0$ in $\phi$ by $P_i + \eps_i \geq 0$ 
				(resp. $P_i \leq 0$ in $\phi$ by $P_i - \eps_i \leq 0$) for $0 \leq i \leq s+1$.
			}
			
			\algstore{myalg}
		\end{algorithmic}
	\end{algorithm}
	
	\begin{algorithm}[H]
		\begin{algorithmic}[1]
			\algrestore{myalg}

			\For {$\mathcal{Q} \subset \mathcal{P}^\star(\eps), \card(\mathcal{Q}) \leq k$}
			\State{
				\[Jac(\mathcal{Q}) \leftarrow  \sum_{1 \leq i_1 < i_2 < \cdots < i_{\card(\mathcal{Q}')} \leq k} \det\left(
				\left( \frac{\partial Q}{\partial X_{i_j}} \right)_{Q \in \mathcal{Q},
					1 \leq i \leq k}\right)
				\]
			}
			\EndFor

			\For {$\mathcal{Q} \subset \mathcal{P}^\star(\eps), \card(\mathcal{Q}) = k+1$}
			\State{
				\[
				\Sigma(\mathcal{Q}') \leftarrow  \sum_{Q \in \mathcal{Q}} Q^2.
				\]
			}
			\EndFor

			\State{
				\[
				\mathcal{H} \leftarrow \{Jac(\mathcal{Q}) \mid \mathcal{Q} \subset \mathcal{P}^\star(\eps), \card(\mathcal{Q}) \leq k\}
				\cup
				\{\Sigma(\mathcal{Q}) \mid \mathcal{Q} \subset \mathcal{P}^\star(\eps), \card(\mathcal{Q}) = k+1\}.
				\]
			} 
			\State{Call Algorithm 14.1 (Block Elimination) from \cite{BPRbook2} with the block of variables $(X_1,\ldots,X_k)$
				and $\mathcal{H}$ as input, and obtain $\mathcal{G} = \mathrm{BElim}_X(\mathcal{F})$ (following the same  notation as  in   \cite[Algorithm 14.1 (Block Elimination)]{BPRbook2}).
			}

			\State{Call Algorithm~\ref{alg:removal} with $\mathcal{G}$ as input and obtain an ordered list
				of Thom encodings $\mathcal{F} = ((f_0,\sigma_0),\ldots,(f_N,\sigma_N))$.}

			\For {$0 \leq i \leq N$}
			\State{ Call Algorithm 14.5 (Quantifier Elimination) \cite{BPRbook2} with input the formula 
				\[
				\widetilde{\psi}(Y) := \forall Z ((f_i(Z) = 0) \wedge (\sign(\Der(f_i))(Z) = \sigma_i)) \Rightarrow (Y \leq Z)
				\] 
				to obtain an equivalent  quantifier-free formula $\widetilde{\psi_i}(Y)$.}

			\State{Call Algorithm~\ref{alg:make-closed} with $\widetilde{\psi_i}(Y)$ as input to obtain a closed formula $\psi_i(Y)$.}
			\State{ $\phi_i \leftarrow \widetilde{\phi} \wedge  \psi_i(Y)$.}
			\State{ $\mathcal{Q}_i \leftarrow \mbox{ the set of polynomials appearing in $\psi_i$}$.}
			\EndFor
			
			\State{Call Algorithm for simplicial replacement with input: 
			the closed formulas $\phi_0,\ldots,\phi_N$, $R$ and $\ell$, 
			and output the simplicial complexes $K_i, 0 \leq i \leq N$.
			}
			
			\COMPLEXITY
			The complexity of the algorithm is bounded by $(s d)^{k^{O(\ell)}}$, where $s = \card(\mathcal{P})$, and
			$d = \max_{P \in \mathcal{P}} \deg(P)$.
			
		\end{algorithmic}
	\end{algorithm}
	\begin{proof}[Proof of correctness]
		The correctness of the algorithm follows from Proposition~\ref{prop:G2}, and the correctness of the following
		algorithms: 
		Algorithm 14.1 (Block Elimination) in \cite{BPRbook2}), Algorithm~\ref{alg:removal},
		Algorithm 14.5 (Quantifier Elimination) in \cite{BPRbook2},  Algorithm~\ref{alg:make-closed}, and
		the Algorithm for simplicial replacement \cite[Theorem 1]{basu-karisani}.
	\end{proof}
	
	\begin{proof}[Complexity analysis]
		The complexity bound follows from the complexity bounds of 
		Algorithm 14.1 (Block Elimination) in \cite{BPRbook2}), Algorithm~\ref{alg:removal},
		Algorithm 14.5 (Quantifier Elimination) in \cite{BPRbook2},  Algorithm~\ref{alg:make-closed}, and
		the Algorithm for simplicial replacement
		\cite[Theorem 1]{basu-karisani}.
	\end{proof}
	
	\subsection{Computing barcodes of semi-algebraic filtrations}
	\label{subsec:algo-sa}
	We can now describe our algorithm for computing the barcode of the filtration of a
	semi-algebraic set by the sub-level sets of a polynomial. 
	First we need an algorithm for computing barcodes of finite filtrations of finite simplicial complexes.

	\begin{algorithm}[H]
		\caption{(Barcode of a finite filtration of finite simplicial complexes)}
		\label{alg:barcode-simplicial}
		\begin{algorithmic}[1]
			\INPUT
			\Statex{
				\begin{enumerate}[1.]
					\item
					$\ell  \in \Z_{\geq 0}$.
					\item
					A finite filtration $\mathcal{F}$, 
					$K_0 \subset \cdots \subset K_N$ 
					of finite simplicial complexes.
				\end{enumerate}
			}
			\OUTPUT
			\Statex{
				$\mathcal{B}_p(\mathcal{F}), 0 \leq p \leq \ell$.
			}

			\PROCEDURE
			\State{ $K_{-1} \leftarrow \emptyset$.}
			\State{ $K_{N+1} \leftarrow K_N$.}
			\For {$-1 \leq i \leq j \leq N+1$}
			\State{Use Gaussian elimination  to compute the persistent Betti numbers $b_p^{i,j}(\mathcal{F})$.}
			\EndFor

		\For{$0 \leq p \leq \ell, 0 \leq i \leq j \leq N+1$}
			\State{ 
				\If{$j = N+1$} 	
				\[	\mu_p^{i,j}  \leftarrow 	b^{i, j}_p
				(\mathcal{F}) - b^{{i-1},{j}}_p (\mathcal{F})
				\]
				\Else
				\[
				\mu_p^{i,j}  \leftarrow  (b^{i, {j-1}}_p (\mathcal{F}) - b^{i, j}_p
				(\mathcal{F})) - (b^{{i-1},{j-1}}_p (\mathcal{F}) - b^{{i-1}, j}_p(\mathcal{F})) \]
				\EndIf					
			}
			
			(cf. Eqn.~\eqref{eqn:prop:finite:multiplicity}).
			\EndFor

			\For {$0 \leq p \leq \ell$} 
			\State{Output
		    $$\displaylines{
			\mathcal{B}_p(\mathcal{F}) = \{(i,j,\mu^{i,j}_p) \;\mid\; 0 \leq i \leq j \leq N, \mu^{i,j}_p > 0\} \cup \cr
			\{(i,\infty,\mu^{i,j}_p) \;\mid\; 0 \leq i \leq j = N+1, \mu^{i,j}_p > 0\}.
			}
			$$
			}
			\EndFor
			\COMPLEXITY
			The complexity of the algorithm is bounded polynomially in $N$ times the number of simplices appearing in the complex $K_N$. 
		\end{algorithmic}
	\end{algorithm}
	
	\begin{proof}[Proof of correctness]
		The correctness of the algorithm follows from  Eqn.~\eqref{eqn:prop:finite:multiplicity}.
	\end{proof}
	
	\begin{proof}[Complexity analysis]
		The complexity of the algorithm follows from the complexity of Gaussian elimination.
	\end{proof}

	\begin{algorithm}[H]
		\caption{(Computing  persistent homology barcodes of semi-algebraic sets)}
		\label{alg:barcode-semi-algebraic}
		\begin{algorithmic}[1]
			\INPUT
			\Statex{
				\begin{enumerate}
					\item
					A $\mathcal{P}$-closed formula $\phi$, with $\mathcal{P}$ a finite subset of  $\D[X_1,\ldots,X_k]$,
					such that $\RR(\phi,\R^k)$ is bounded.
					\item
					A polynomial $P \in \D[X_1,\ldots,X_k]$. 
					\item 
					$\ell, 0 \leq \ell \leq k$.
				\end{enumerate}
			}
			\OUTPUT
			\Statex{
				For each $p, 0 \leq p \leq \ell$, $\mathcal{B}_p(S,P)$,
				where $S = \RR(\phi) 
				$.
			}

			\algstore{myalg}
		\end{algorithmic}
	\end{algorithm}
	
	\begin{algorithm}[H]
		\begin{algorithmic}[1]
			\algrestore{myalg}

			\PROCEDURE
			\State{$\mathcal{P}' \leftarrow  \mathcal{P} \cup \{\eps(X_1^2+\cdots+X_k^2) -1\}$.}
			\State{$\phi' \leftarrow \phi \wedge \eps^2(X_1^2+\cdots+X_k^2)- 1 \leq 0)$.}

			\State{$\R \leftarrow \R\la\eps\ra$, $\D \leftarrow \D[\eps]$.}
			\State{Call Algorithm~\ref{alg:filtration} with input $\ell, 1/\eps, \mathcal{P}',\phi', P$,
				to obtain a finite ordered set
				of Thom encodings $(f_0,\sigma_0),\ldots, (f_N,\sigma_N)$, and a finite filtration $\mathcal{F} = (K_0 \subset \cdots \subset K_N)$, where $K_N$ is a finite simplicial complex.}

			\State{Call Algorithm~\ref{alg:barcode-simplicial} with input $\ell$ and the finite filtration 
			$\mathcal{F}$, 
				and output for each $p, 0 \leq p \leq \ell$, 
				$\mathcal{B}_p(\mathcal{F})$.}

			\State {{\bf for each} $p, 0\leq p \leq \ell$
							
				Output
				$$\displaylines{
					\mathcal{B}_p(S,P) = \bigcup_{(i,j,\mu) \in \mathcal{B}_p(\mathcal{F}),0 \leq i \leq j \leq N}
					\{((f_i,\sigma_i),(f_j,\sigma_j),\mu)\}
					\cup \cr
					\bigcup_{(i,\infty,\mu) \in \mathcal{B}_p(\mathcal{F})}\{((f_i,\sigma_i),\infty,\mu)\}.
				}
				$$
			}

			\COMPLEXITY
			The complexity of the algorithm is bounded by $(s d)^{k^{O(\ell)}}$, where $s = \card(\mathcal{P})$, and
			$d = \max_{Q \in \mathcal{P} \cup \{P\}} \deg(Q)$.
		\end{algorithmic}
	\end{algorithm}
	
	\begin{proof}[Proof of correctness]
		The correctness of the algorithm follows from the correctness of Algorithms~\ref{alg:filtration} and 
		\ref{alg:barcode-simplicial}.
	\end{proof}
	
	\begin{proof}[Complexity analysis]
		The complexity bound follows from the complexity bounds of Algorithms~\ref{alg:filtration} and 
		\ref{alg:barcode-simplicial}.
	\end{proof}
	
	\begin{proof}[Proof of Theorem~\ref{thm:persistent}]
		The theorem follows from the correctness and the complexity analysis of Algorithm~\ref{alg:barcode-semi-algebraic}.
	\end{proof}

	\section{Future work and open problems}
	\label{sec:conclusion}
	We conclude by stating some open problems and possible future directions of research in this area.
	\begin{enumerate}[1.]
	    \item
	    It would be very interesting (and challenging) to obtain an algorithm with singly exponential complexity that computes the entire barcode of a semi-algebraic filtration, and not restricted to dimension up to $\ell$. This would imply also an algorithm with singly exponential complexity for computing all the Betti numbers of a given semi-algebraic set, which is a challenging problem on its own \cite{Basu-survey}. 
	    \item 
	    Another open problem is to extend Algorithm~\ref{alg:barcode-semi-algebraic} to the case of non-proper semi-algebraic maps using the proposed definition of barcodes for non-proper semi-algebraic maps (see  Definition~\ref{def:tilde-B}).
		\item
		One very active topic in the area of persistent homology is the theory of multi-dimensional
		persistent homology \cite{Oudot}. In our setting this would imply studying the sub-level sets of two or more real polynomial functions simultaneously. While the so called persistence modules and associated
		barcodes can be defined analogously to the one-dimensional situation (see for example \cite{Oudot}), an analog of Proposition~\ref{prop:finite} is missing. It is thus an open problem to give an algorithm with singly exponential complexity to compute the barcodes of ``higher dimensional'' semi-algebraic filtrations.
	\end{enumerate} 
	
\section*{Acknowledgements}
The authors are grateful to Ezra Miller for his comments on a previous version of this paper and for pointing out several related prior works.

	\bibliographystyle{amsplain}
	\bibliography{master}
\end{document}